\documentclass[11pt]{amsart}
\usepackage{latexsym, amssymb, amsmath, amscd, amsfonts}
\usepackage[mathscr]{eucal}
\usepackage{mathrsfs}
\usepackage{color}
\usepackage{concrete,euler,textcomp}
\usepackage[T1]{fontenc}
\usepackage{graphicx}
\usepackage{geometry}

\newtheorem{theorem}{Theorem}[section]
\newtheorem*{acknowledgement}{Acknowledgement}

\newtheorem{corollary}[theorem]{Corollary}

\newtheorem{definition}[theorem]{Definition}

\newtheorem{lemma}[theorem]{Lemma}

\newtheorem{proposition}[theorem]{Proposition}

\newtheorem{remark}[theorem]{Remark}

\newcommand{\f}[1]{\mathfrak{#1}}
\newcommand{\calli}[1]{\mathcal{#1}}
\newcommand{\bbC}{\mathbb{C}}
\newcommand{\D}[1]{\Lambda_{#1}}
\newcommand{\B}{\mathcal{B}}
\input xy
\xyoption{all}
\input{tcilatex}

\begin{document}




\title[Symplectic multiplicity spaces]{An analysis of the multiplicity spaces in branching of symplectic groups}
\author{Oded Yacobi}
\email{oyacobi@math.tau.ac.il}
\address{School of Mathematical Sciences, Tel Aviv University, Tel Aviv
69978, Israel}
\subjclass[2000]{20G05, 05E15}



\begin{abstract}
Branching of symplectic groups is not multiplicity-free.  We
describe a new approach to resolving these multiplicities that is
based on studying the associated branching algebra $\B$. The algebra
$\B$ is a graded algebra whose components encode the multiplicities
of irreducible representations of $Sp_{2n-2}$ in irreducible
representations of $Sp_{2n}$. Our first theorem states that the map
taking an element of $Sp_{2n}$ to its principal $n \times (n+1)$
submatrix induces an isomorphism of $\B$ to a different branching
algebra $\B'$.  The algebra $\B'$ encodes multiplicities of
irreducible representations of $GL_{n-1}$ in certain irreducible
representations of $GL_{n+1}$. Our second theorem is that each
multiplicity space that arises in the restriction of an irreducible
representation of $Sp_{2n}$ to $Sp_{2n-2}$ is canonically an
irreducible module for the $n$-fold product of $SL_{2}$.  In
particular, this induces a canonical decomposition of the
multiplicity spaces into one dimensional spaces, thereby resolving
the multiplicities.

\end{abstract}

\subjclass[2000]{20G05, 05E10}
\keywords{symplectic group, branching
algebra}
\maketitle







\tableofcontents

\section{Introduction \label{sectionone}}
The purpose of this paper is to give a new interpretation of
symplectic branching which, unlike the branching for the other
towers of classical groups, is not multiplicity free.  In other
words, an irreducible representation of $Sp_{2n}$ does not decompose
uniquely into irreducible representations of $Sp_{2n-2}$ (embedded
as the subgroup fixing pointwise a two-dimensional non-isotropic
subspace). We resolve this ambiguity by analyzing the algebraic
structure of the associated multiplicity spaces.

Our main object of study is an algebra $\B$ associated to the pair
$(Sp_{2n-2},Sp_{2n})$.  This ``branching algebra'' is a graded
algebra whose components are all the multiplicity spaces that appear
in the restriction of irreducible representations of $Sp_{2n}$ to
$Sp_{2n-2}$.  By definition $\B$ is a certain subalgebra of the ring of
regular functions, $\calli{O}(Sp_{2n})$, on $Sp_{2n}$:
$$
\B=\calli{O}(\overline{U}_{C_{n}} \setminus Sp_{2n} \text{ } /
\text{ } U_{C_{n-1}}).
$$
(See Section \ref{section2.2} for notation.)  From this realization
it follows that $\B$ has a natural action of $SL_{2}$ by right
translation.

Our first result relates $\B$ to a different branching algebra associated
to the branching pair $(GL_{n-1},GL_{n+1})$.  Using
$(GL_{n-1},GL_{n+1})$-duality we define the algebra
$$
\B'=\calli{O}(\overline{U}_{n} \setminus M_{n,n+1} \text{ } / \text{
} U_{n-1}).
$$
This algebra is a ``restricted'' branching algebra, as it is
isomorphic to a direct sum of only certain multiplicity spaces that
occur in branching from $GL_{n+1}$ to $GL_{n-1}$.  Note that $\B'$
is also a graded $SL_{2}$-algebra. We consider the function
$\psi:Sp_{2n} \rightarrow M_{n,n+1}$ which maps an element of
$Sp_{2n}$ to its $n \times (n+1)$ principal submatrix. The theorem is
that the induced map on functions $\psi^{*}:\calli{O}(M_{n,n+1})
\rightarrow \calli{O}(Sp_{2n})$ restricts to give an isomorphism
$\psi^{*}:\B' \rightarrow \B$ of graded $SL_{2}$-algebras.

This theorem allows us to reduce questions about branching from
$Sp_{2n}$ to $Sp_{2n-2}$ to analogous questions concerning branching
from $GL_{n+1}$ to $GL_{n-1}$.  The latter are easier, as they can
be ``factored'' through $GL_{n}$.  We will illustrate this reduction
technique several times, most notably in order to prove our second
theorem.

To describe our second theorem we introduce a family of subalgebras
of $\calli{B}$ indexed by a finite set, $\Sigma$, of so-called order
types. We prove that each subalgebra $\calli{B}_{\sigma}$ is
isomorphic to the algebra, $\calli{O}(V)$, of polynomials on a
vector space $V$. This isomorphism is unique up to scalars.
Moreover, $V$ can be given the structure of an
$L=\prod_{i=1}^{n}SL_{2}$-module. Therefore, via this isomorphism,
we obtain a canonical action of $L$ on $\calli{B}_{\sigma}$ by
algebra automorphisms.  The action of $L$ is well-defined on the
intersections of these subalgebras, allowing us to glue the modules
together to obtain a representation, $\Phi$, of $L$ on $\calli{B}$.

The representation $(\Phi, \calli{B})$ of $L$ satisfies some
remarkable properties.  First and foremost, it identifies each
multiplicity space as an explicit irreducible $L$-module.  Secondly,
the restriction of $\Phi$ to the diagonal subgroup of $L$ recovers
the natural $SL_{2}$ action on $\calli{B}$. Finally, it is the
unique such representation acting by algebra automorphisms on the
subalgebras $\calli{B}_{\sigma}$.

As a corollary of this theorem we obtain resolution of the
multiplicities that occur in branching of symplectic groups. Indeed,
irreducible $L$-modules have one dimensional weight spaces.
Therefore, via $\Phi$, we obtain a decomposition of the multiplicity
spaces into one dimensional spaces.  This decomposition is
canonical, i.e. depends only the choice of torus of $Sp_{2n}$ that
is fixed throughout.  The other known approach to this problem uses
quantum groups, where these multiplicities are resolved using an
infinite dimensional Hopf algebra called the twisted Yangian
(\cite{M2}).

The basis of $\B$ which we obtain from the action of $L$ is unique
up to scalar.  In \cite{KY} we study properties of this basis, and,
in particular, show that it is a standard monomial basis, i.e. it
satisfies a straightening algorithm.  By induction this basis can be used to obtain a basis for
irreducible representations of $Sp_{2n}$.  The resulting basis is a partial analogue of
the Gelfand-Zetlin basis (\cite{GZ1}); to be properly called a ``Gelfand-Zetlin'' basis one must also
compute the action of Chevalley generators.  The only such basis known is the Gelfand-Zetlin-Molev basis arising from
the Yangian theory mentioned above (\cite{M2}).  In a future work we will compare the basis resulting from our
work to the Gelfand-Zetlin-Molev basis.

%
%
%

\begin{acknowledgement}
The results herein are based on the author's UC San Diego PhD thesis
(2009).  The author is grateful to his advisor, Nolan Wallach, for
his guidance and insights.  The author also thanks Avraham Aizenbud,
Sangjib Kim, Allen Knutson and Gerald Schwarz for helpful
conversations. This work was supported in part by the ARCS
Foundation.
\end{acknowledgement}

\section{Preliminaries}
Our main object of study, $\calli{B}$, is an example of a branching
algebra. In section \ref{section2.1} we define branching algebras
and their associated branching semigroups.  In section
\ref{section2.2} we fix some notation that will be used throughout.
\subsection{\label{section2.1} Branching algebras}
Let $G$ be a connected
classical group with identity $e\in G$. \ Fix a maximal torus $T_{G}$, Borel
subgroup $B_{G}$, and unipotent radical $U_{G}\subset B_{G}$ so that $T_{G}U_{G}=B_{G}$. \ Let $%
\overline{U}_{G}$ be the unipotent group opposite $U_{G}$.
When convenient, we work in the setting of Lie algebras. \ We
denote the complex Lie algebra of a complex Lie group by the corresponding
lower-case fraktur letter.

Fix the choice of positive roots giving $U_{G}$: $\Phi ^{+}_{G}=\Phi
(\f{b}_{G},\f{t}_{G})$. \ Let $\Lambda_{G} \subset \f{t}^{*}$ be the
corresponding semigroup of dominant integral weights. \ Denote by
$F^{\lambda }_{G}$ the finite-dimensional irreducible representation
of $G$ of highest weight $\lambda \in \Lambda_{G}$.

For an affine algebraic variety $X$, let $%
\mathcal{O}(X)$ denote the algebra of regular functions on $X$.
The group $G$ has the structure of an affine algebraic variety, and
$\mathcal{O}(G)$ is a $G\times G$-module under left and right
translation.  Let $\calli{R}_{G}=\mathcal{O}(\overline{U}_{G}
\backslash G)$ be the left $\overline{U}_{G}$-invariant functions on
$G$:
$$
\calli{R}_{G} = \{ f \in \calli{O}(G) : f(\overline{u}g)=f(g) \text{
for all } \overline{u} \in \overline{U}_{G} \text { and } g \in G
\}.
$$
Under the right action of $G$ (see Theorem 4.2.7. ,\cite{GW}):
\begin{equation*}
\calli{R}_{G} \cong \bigoplus_{\lambda \in
\Lambda_{G}}((F^{\lambda}_{G})^{*})^{\overline{U}_{G}}\otimes
F^{\lambda }_{G}\cong \bigoplus_{\lambda \in \Lambda_{G}}F^{\lambda
}_{G}\text{.}
\end{equation*}%


Henceforth identify $%
F^{\lambda }_{G}$ with its image in $\calli{R}_{G}$.
Let $f^{\lambda }_{G} \in F^{\lambda }_{G}$ be the unique highest
weight vector such that $f^{\lambda }_{G}(e)=1$. We call $f^{\lambda
}_{G}$ the \textbf{canonical highest weight vector} of
$F^{\lambda}_{G}$.
Now let $\lambda ,\lambda ^{\prime }\in \Lambda_{G}$.  Then $f^{\lambda }_{G}f^{\lambda ^{\prime }}_{G}$ is $U$%
-invariant of weight $\lambda +\lambda ^{\prime }$. \ Since also
$f^{\lambda }_{G}f^{\lambda ^{\prime }}_{G}(e)=1$, it follows that
$f^{\lambda }_{G}f^{\lambda
^{\prime }}_{G}=f^{\lambda +\lambda ^{\prime }}_{G}$, and therefore $%
F^{\lambda }_{G}F^{\lambda ^{\prime }}_{G}=F^{\lambda +\lambda
^{\prime }}_{G}$.

The \textbf{Cartan product},
$\pi _{\lambda ,\lambda ^{\prime }}:F^{\lambda }_{G}\otimes
F^{\lambda ^{\prime }}_{G}\rightarrow F^{\lambda +\lambda ^{\prime
}}_{G}$,
is defined by $\pi _{\lambda ,\lambda ^{\prime }}(v\otimes v^{\prime
})=vv^{\prime }$. The \textbf{Cartan embedding},
$j_{\lambda ,\lambda ^{\prime }}:F^{\lambda +\lambda ^{\prime
}}_{G}\rightarrow F^{\lambda }_{G}\otimes F^{\lambda ^{\prime
}}_{G}$,
is given by setting $j_{\lambda ,\lambda ^{\prime }}(f^{\lambda
+\lambda ^{\prime
}}_{G})=f^{\lambda }_{G}\otimes f^{\lambda ^{\prime }}_{G}$, and extending by $G$%
-linearity.

Suppose now that $H \subset G$ is a connected Lie subgroup, and we
have chosen its distinguished subgroups so that $\overline{U}_{H}
\subset \overline{U}_{G}$, $T_{H} \subset T_{G}$, and $U_{H} \subset
U_{G}$.
Consider the subalgebra of bi-invariants
\begin{equation*}
\calli{B}(H,G)=\calli{O}(\overline{U}_{G}\setminus G\text{ }/\text{
}U_{H}) \subset \calli{R}_{G}.
\end{equation*}
In other words, $\B(H,G)$ consists of the functions $f \in
\calli{R}_{G}$ such that $f(gu)=f(g)$ for all $g \in G$ and $u \in
U_{H}$.

Since $T_{H}$ normalizes $U_{H}$ there is an action of $T_{H}$ on
$\calli{B}(H,G)$ by right translation.  The decomposition of
$\calli{B}(H,G)$ into $T_{H}$ weight spaces is:
\begin{equation*}
\label{summultspaces} \calli{B}(H,G) = \bigoplus_{(\mu,\lambda) \in
\Lambda_{H} \times \Lambda_{G}} F^{\lambda / \mu}.
\end{equation*}
The weight space $F^{\lambda / \mu}$ is defined as
\begin{equation}
\label{multspacedef} F^{\lambda / \mu} = \{ f \in
(F^{\lambda}_{G})^{U_{H}} : f(gt)=t^{\mu}f(g) \text{ for all } g \in
G \text{ and } t \in T_{H} \}.
\end{equation}
Equivalently, $F^{\lambda / \mu}$ is the $T_{G} \times T_{H}$ weight
space of $\B$ corresponding to the weight $(-\lambda)$ for $T_{G}$
and $\mu$ for $T_{H}$.

Note that the centralizer $Z_{G}(H)$ acts on $F^{\lambda / \mu}$ by
right translation.  As a $Z_{G}(H)$-module there is a canonical
isomorphism
$$
F^{\lambda / \mu} \cong Hom_{H}(F^{\mu}_{H},F^{\lambda}_{G}).$$ (The
isomorphism maps $\phi \in Hom_{H}(F^{\mu}_{H},F^{\lambda}_{G})$ to
$\phi(f^{\mu}_{H}) \in F_{\lambda / \mu}$.)  In particular, the
dimension of $F^{\lambda / \mu}$ counts the branching multiplicity
of $F^{\mu}_{H}$ in the restriction to $H$ of $F^{\lambda}_{G}$. For
this reason $\calli{B}(H,G)$ is termed the \textbf{branching
algebra} for the pair $(H,G)$ (cf. \cite{HTW}, \cite{Z} and
references therein).
\begin{lemma}
\label{LemmaGrade}Let $(\mu,\lambda), (\mu',\lambda') \in
\Lambda_{H} \times \Lambda_{G}$. Then
\begin{equation*}
\pi _{\lambda ,\lambda ^{\prime }}(F^{\lambda / \mu} \otimes
F^{\lambda' / \mu'}) \subset F^{\lambda+\lambda' / \mu+\mu'}.
\end{equation*}
\end{lemma}
\begin{proof}
Since $\pi _{\lambda ,\lambda ^{\prime }}$ is a $G$-module morphism,
in particular it is a $U_{H}$-module morphism. \ Therefore, $\pi
_{\lambda ,\lambda ^{\prime }}((F^{\lambda }_{G})^{U_{H}}\otimes
(F^{\lambda ^{\prime }}_{G})^{U_{H}})\subset (F^{\lambda +\lambda
^{\prime }}_{G})^{U_{H}}$. \ Now note that $\pi _{\lambda ,\lambda
^{\prime }}$ intertwines the $T_{H}$ action on $(F^{\lambda
}_{G})^{U_{H}} \otimes (F^{\lambda ^{\prime }})_{G}^{U_{H}}$ and
$(F^{\lambda +\lambda ^{\prime }}_{G})^{U_{H}} $.\ Let $((F^{\lambda
}_{G})^{U_{H}} \otimes (F^{\lambda ^{\prime }}_{G})^{U_{H}}))(\mu
+\mu ^{\prime })$ be the $\mu +\mu ^{\prime }$ weight space of
$T_{H}$ on $(F^{\lambda }_{G})^{U_{H}} \otimes (F^{\lambda ^{\prime
}}_{G})^{U_{H}})$.  Then $\pi _{\lambda ,\lambda ^{\prime }}$ maps
$((F^{\lambda }_{G})^{U_{H}} \otimes (F^{\lambda ^{\prime
}}_{G})^{U_{H}}))(\mu +\mu ^{\prime })$ to the weight space
$F^{\lambda +\lambda ^{\prime } / \mu +\mu ^{\prime }}$. \ Since
$F^{\lambda / \mu} \otimes F^{\lambda' / \mu'} \subset ((F^{\lambda
}_{G})^{U_{H}} \otimes (F^{\lambda ^{\prime }}_{G})^{U_{H}})(\mu
+\mu ^{\prime })$ the proof is complete.
\end{proof}

By the lemma $\calli{B}(H,G)$ is a $\Lambda_{H} \times
\Lambda_{G}$-graded algebra.  Abusing notation a bit, we denote the
restriction of $\pi _{\lambda ,\lambda ^{\prime }}$ to $F^{\lambda
/\mu} \otimes F^{\lambda' / \mu'}$ also by $\pi _{\lambda ,\lambda
^{\prime }}$. \ This will cause no confusion since we will
explicitly write
\begin{equation*}
F^{\lambda / \mu} \otimes F^{\lambda' / \mu'}\overset{\pi _{\lambda
,\lambda ^{\prime }}}{\rightarrow} F^{\lambda + \lambda' / \mu +
\mu'}
\end{equation*}
when referring to this map, which we call the \textbf{Cartan product
of multiplicity spaces}. In general the Cartan product of
multiplicity spaces is not surjective (see section Proposition
\ref{SpnSurjectivityCorollary}). This observation will be critical.

We can associate to the pair $(H,G)$ the following set:
$$
\Lambda_{\B(H,G)}=\{(\mu,\lambda) \in \Lambda_{H} \times \Lambda_{G}
: F^{\lambda / \mu} \neq \{0\} \}.
$$
\begin{lemma}
\label{branchingsemigrouplemma} The set $\Lambda_{\B(H,G)}$ is a
semigroup under entry-wise addition.
\end{lemma}
\begin{proof}
Suppose $(\mu,\lambda),(\mu',\lambda') \in \Lambda_{\B(H,G)}$.
Choose $0 \neq x \in F^{\lambda / \mu}$ and $0 \neq x' \in
F^{\lambda' / \mu'}$. Since $\calli{B}(H,G)$ has no zero divisors,
$xx' \neq 0$. By the above lemma this implies that $F^{\lambda +
\lambda' / \mu + \mu'} \neq 0$, i.e. $(\mu+\mu',\lambda+\lambda')
\in \Lambda_{\B(H,G)}$. The other semigroup axioms are trivial to
check.
\end{proof}

Thus the $\Lambda_{H} \times \Lambda_{G}$-graded algebra
$\calli{B}(H,G)$ can also be regarded as a
$\Lambda_{\B(H,G)}$-graded algebra, and we call $\Lambda_{\B(H,G)}$
a  \textbf{branching semigroup}.

\subsection{\label{section2.2} Notation}
Henceforth fix an integer $n>1$. Let $\Lambda _{n}$ be the semigroup
of weakly decreasing sequences of length $n$ consisting of
non-negative integers.
For $m \geq 1$, let
$\Lambda_{m,n}=\Lambda_{m} \times \Lambda_{n}$.

The set of dominant weights
for irreducible polynomial representations of $GL_{n}=GL(n,\bbC)$
is identified with $%
\Lambda _{n}$ in the usual way (see e.g. Theorem 5.5.22.,
\cite{GW}). \ For $\lambda \in \Lambda _{n}$, let $V^{\lambda
}=F^{\lambda}_{GL_{n}}$ be the irreducible representation of
$GL_{n}$ with highest weight $\lambda $, which we realize in
$\mathcal{R}_{GL_{n}}$ as described in Section \ref{section2.1}. Let
$v_{\lambda} \in V^{\lambda }$ be the canonical highest weight
vector.

Let $T_{n}$ be the subgroup of
diagonal matrices in $GL_{n}$, $U_{n}$ the subgroup of
upper-triangular unipotent matrices, and $\overline{U}_{n}$ be the subgroup
of lower-triangular unipotent matrices.  Suppose $1\leq m<n$. \ We embed $GL_{m}$ in $%
GL_{n}$ as the subgroup:
\begin{equation*}
\left\{
\begin{bmatrix}
g &  \\
& I_{n-m}%
\end{bmatrix}%
:g\in GL_{m}\right\} \text{.}
\end{equation*}%

For $(\mu,\lambda) \in \Lambda _{m,n}$, define the multiplicity
space $V^{\lambda / \mu }$ as in (\ref{multspacedef}) with
$H=GL_{m}$ and $G=GL_{n}$.
In the case $m=n-1$ these spaces are classically known (see e.g.
Theorem 8.1.1., \cite{GW}): for $(\mu,\lambda) \in \Lambda
_{n-1,n}$,
\begin{eqnarray}
\label{Eq1.1}
\dim V^{\lambda / \mu } &\leq& 1 \\
\label{Eq1.2} V^{\lambda / \mu }\not=\{0\} &\Leftrightarrow& \mu
\textbf{ interlaces } \lambda.
\end{eqnarray}
The interlacing condition, written $\mu < \lambda$, means that $\lambda_{i}\geq \mu_{i} \geq \lambda_{i+1}$
for $i=1,...,n-1$, where $\mu=(\mu_{1},...,\mu_{n-1})$ and $\lambda=(\lambda_{1},...,\lambda_{n})$.

Let $M_{m,n}$ denote the space of $m \times n$ matrices with complex
entries. By $(GL_{n},GL_{n+1})$ duality, the space of polynomials
$\calli{O}(M_{n,n+1})$ decomposes multiplicity free as $GL_{n}
\times GL_{n+1}$ module (see e.g. Theorem 5.6.7.,\cite{GW}):
$$
\calli{O}(M_{n,n+1}) \cong \bigoplus_{\lambda \in
\Lambda_{n}}V^{\lambda^{*}}\otimes V^{\lambda^{+}}.
$$
Here $\lambda^{+} \in \Lambda_{n+1}$ is obtained from $\lambda$ by
adding a zero.  Notice that $V^{\lambda^{*}}$ is an irreducible
representation of $GL_{n}$, while $V^{\lambda^{+}}$ is an
irreducible representation of $GL_{n+1}$.

We will study the algebra of bi-invariants
$$
\calli{B}'=\calli{O}(\overline{U}_{n} \setminus M_{n,n+1}\text{
}/\text{ }U_{n-1}) \cong \bigoplus_{ (\mu,\lambda) \in
\Lambda_{n-1,n}} V^{\lambda^{+} / \mu}.
$$
This algebra is a sort of ``restricted'' branching algebra as it includes only certain
multiplicity spaces that occur in branching from $GL_{n+1}$ to $GL_{n-1}$.

The algebra $\B'$ is an $\Lambda_{n-1,n}$-graded algebra; the
$(\mu,\lambda)$ component is $V^{\lambda^{+} / \mu}$. Moreover, it
is naturally an $SL_{2}$-algebra. Indeed, there is an obvious
$SL_{2} \subset GL_{n+1}$ that commutes with $GL_{n-1}$, and
therefore acts on $\calli{B}'$ by right translation. This action
clearly preserves the graded components of $\calli{B}'$, i.e. the
multiplicity spaces $V^{\lambda^{+} / \mu}$ are naturally
$SL_{2}$-modules.

By Lemma \ref{branchingsemigrouplemma}, the algebra $\calli{B}'$ is
also graded by the semigroup
$$
\Lambda_{\B'}=\{ (\mu,\lambda) \in \Lambda_{n-1,n} : V^{\lambda^{+}
/ \mu} \neq \{0\} \}.
$$
It will be useful for us to have a more concrete realization of this
semigroup.  Suppose $\mu=(\mu_{1},...,\mu_{n-1}) \in \D{n-1}$ and
$\lambda=(\lambda_{1},...,\lambda_{n+1}) \in \D{n+1}$. We say $\mu$
\textbf{double interlaces} $\lambda$, written $\mu \ll \lambda$, if
for $i=1,...,n-1$,
\begin{equation}
\label{Eq1.2.1} \lambda_{i} \geq \mu_{i} \geq \lambda_{i+2}.
\end{equation}
If $(\mu,\lambda) \in \Lambda_{n-1,n}$, then $\mu$ double interlaces
$\lambda$, also written $\mu \ll \lambda$, if $\mu \ll \lambda^{+}$.

Suppose $(\mu,\lambda) \in \Lambda_{n-1,n}$.  It follows that $\mu
\ll \lambda$ if, and only if, there exists $\gamma \in \D{n}$ such
that $\mu < \gamma < \lambda^{+}$ (see Lemma \ref{simplelemma}).
Therefore by (\ref{Eq1.2}),
\begin{equation}
\label{DefE} \Lambda_{\B'}=\{(\mu ,\lambda ) \in \Lambda _{n-1,n} :
\mu \ll \lambda \}.
\end{equation}

\medskip

Next we consider the symplectic groups.  Label a basis for $\mathbb{C}^{2n}$ as $e_{\pm 1},...,e_{\pm n}$ where $%
e_{-i}=e_{2n+1-i}$.  Denote by $s_{n}$ the $n\times n$ matrix with one's on the anti-diagonal and
zeros everywhere else. \ Set
\begin{equation*}
J_{n}=%
\begin{bmatrix}
0 & s_{n} \\
-s_{n} & 0%
\end{bmatrix}%
\end{equation*}%
and consider the skew-symmetric bilinear form $\Omega
_{n}(x,y)=x^{t}J_{n}y$ on $\mathbb{C}^{2n}$. \ Define the symplectic
group relative to this form: $Sp_{2n}=Sp(\bbC^{2n},\Omega)$. In this
realization we can take as a maximal torus $T_{C_{n}} = T_{2n}\cap
Sp_{2n}$, a maximal unipotent subgroup $U_{C_{n}} = U_{2n}\cap
Sp_{2n}$, and its opposite $\overline{U}_{C_{n}} =
\overline{U}_{2n}\cap Sp_{2n}$

Embed $Sp_{2n-2}$ in $Sp_{2n}$ as the subgroup fixing the vectors
$e_{\pm n}$.  Notice that
\begin{equation}
\label{Sp1} \{ g \in Sp_{2n} : ge_{\pm i}=e_{\pm i}\text{ for
}i=1,...,n-1 \}
\end{equation}
is a subgroup isomorphic to $SL_{2}=Sp_{2}$ commuting with $Sp_{2n-2}$.

The set of dominant weights for $Sp_{2n}$ is identified with
$\Lambda_{n}$ (see e.g. Theorem 3.1.20., \cite{GW}). \ For $\lambda
\in \Lambda_{n}$, let $W^{\lambda }=F^{\lambda}_{Sp_{2n}}$
be the irreducible representation of $Sp_{2n}$ with highest weight $%
\lambda $, which we realize in $\mathcal{R}_{Sp_{2n}}$ as described
in section \ref{section2.1}. Let $w_{\lambda} \in W^{\lambda}$ be
the canonical highest weight vector.

For $(\mu,\lambda) \in \Lambda_{n-1,n}$ we define the multiplicity
space $W^{\lambda / \mu}$ as in (\ref{multspacedef}) with
$H=Sp_{2n-2}$ and $G=Sp_{2n}$. Therefore the dimension of
$W^{\lambda / \mu}$ is the multiplicity of the irreducible
representation $W^{\mu}$ of $Sp_{2n-2}$ in the representation
$W^{\lambda}$ of $Sp_{2n}$.  In contrast to the general linear
groups, the branching of the symplectic groups is not
multiplicity-free, i.e. $dim W^{\lambda / \mu}>1$ for generic
$(\mu,\lambda) \in \Lambda_{n-1,n}$ (see Corollary
\ref{CorollaryWY}).

Our main object of study is the following branching algebra:
\begin{eqnarray*}
\calli{B}=\calli{B}(Sp_{2n-2},Sp_{2n}).
\end{eqnarray*}
By Lemma \ref{LemmaGrade}, $\B$ is graded by $\Lambda_{n-1,n}$:
$$
\B=\bigoplus_{(\mu,\lambda) \in \Lambda_{n-1,n}}W^{\lambda / \mu}.
$$
By Lemma \ref{branchingsemigrouplemma}, we also consider $\B$ as
graded over the branching semigroup
$$\Lambda_{\B}=\Lambda(Sp_{2n-2},Sp_{2n}).$$

The branching algebra $\B$ is naturally an $SL_{2}$-module for the
copy of $SL_{2}$ appearing in (\ref{Sp1}). Indeed, this copy
$SL_{2}$ acts on $\B$ by right translation leaving the graded
components $W^{\lambda / \mu}$ invariant.  In other words, the
multiplicity spaces $W^{\lambda / \mu}$ are naturally
$SL_{2}$-modules.  We refer to this action as the ``natural''
$SL_{2}$ action, and denote it simply by $x.b$ for $x \in SL_{2}$
and $b \in \B$.

\medskip
Finally, we consider the group $SL_{2}=SL(2,\bbC)$. \ Let
$F^{k}=\mathcal{O}^{k}(\mathbb{C}^{2})$ be the
$(k+1)^{th}$-dimensional irreducible representation of $SL_{2}$,
realized as the polynomials on $\mathbb{C}^{2}$ of homogeneous
degree $k$.  The $SL_{2}$ action is by right translation. \ For
$k,k^{\prime }\geq 0$ let $\pi _{k,k^{\prime }}:F^{k}\otimes
F^{k^{\prime }}\rightarrow F^{k+k^{\prime }}$ be the usual
multiplication of functions, and define the embedding of
$SL_{2}$-modules, $j_{k,k^{\prime }}:F^{k+k^{\prime }}\rightarrow
F^{k}\otimes F^{k^{\prime }}$, by setting $j_{k,k^{\prime
}}(x_{1}^{k+k^{\prime }})=x_{1}^{k}\otimes x_{1}^{k^{\prime }}$ and
extending by $SL_{2}$-linearity.  Here $x_{1}$ is the first
coordinate function on $\bbC^{2}$.

\begin{remark}
We use the symbols $\pi _{\lambda ,\lambda ^{\prime }}$ and $j _{\lambda ,\lambda ^{\prime }}$ to
denote the Cartan maps for any of the given groups above.  It will be
clear from context which group we have in mind.
\end{remark}

\section{Main Results \label{section1.2}}
In this section we describe in detail the two main theorems of this
paper.

\subsection{An isomorphism of branching algebras}
Our first theorem describes an isomorphism of the branching algebras
$\B$ and $\B'$. This usefulness of this theorem will become clear,
as we use it repeatedly to reduce questions about branching of the
symplectic groups to analogous questions about branching of the
general linear groups.

Define the map $ \psi:Sp_{2n} \rightarrow M_{n,n+1} $, which assigns
an element $g \in Sp_{2n}$ its principal $n \times (n+1)$ submatrix.
Consider the induced map on functions
\begin{equation}
\label{psistar} \psi^{*}:\calli{O}(M_{n,n+1}) \rightarrow
\calli{O}(Sp_{2n}).
\end{equation}
In Lemma \ref{psistarlemma} we show that $\psi^{*}(\B') \subset \B$,
and, moreover, that $\psi^{*}: \B' \rightarrow \B$ is a map of
$\Lambda_{n-1,n}$-graded, $SL_{2}$ algebras. In fact, much more is
true:

\begin{theorem}
\label{TransferThm2}The map $\psi^{*}: \B' \rightarrow \B$ is an isomorphism of $\Lambda_{n-1,n}$-graded, $SL_{2}$-algebras.
\end{theorem}

We now describe some applications.

For our first application let $(\mu ,\lambda ) \in \Lambda
_{n-1,n}$.  By the theorem,
$$W^{\lambda / \mu} \neq \{0\} \Leftrightarrow V^{\lambda^{+} / \mu}
\neq \{0\}.$$  Combining this with \eqref{DefE}, we recover a
classical result about symplectic branching (see e.g. Theorem
8.1.5., \cite{GW}):
\begin{equation*}
\label{Eq1.2.2} W^{\lambda / \mu}\neq \{0\} \Leftrightarrow \mu \ll
\lambda,
\end{equation*}
i.e. $\Lambda_{\B}=\Lambda_{\B'}$.

To describe our second application of this theorem, recall first
that the multiplicity spaces $W^{\lambda / \mu}$ are each
$SL_{2}$-modules.  Naturally, one would like to describe the
$SL_{2}$-module structure of these multiplicity spaces. By Theorem
\ref{TransferThm2}, $W^{\lambda / \mu} \cong V^{\lambda^{+} / \mu}$
as $SL_{2}$-modules, so it suffices to answer the analogous question
for the general linear groups.  But this is not too difficult, since
branching from $GL_{n+1}$ to $GL_{n-1}$ factors through $GL_{n}$.

Given $(\mu ,\lambda ) \in \Lambda _{n-1,n+1}$ let $(x_{1} \geq
y_{1} \geq \cdots \geq x_{n} \geq y_{n})$ be the non-increasing
rearrangement of
$(\mu_{1},...,\mu_{n-1},\lambda_{1},...,\lambda_{n+1})$.  Set
$$
r_{i}(\mu,\lambda)=x_{i}-y_{i}.
$$

\begin{proposition}
\label{Prop1.2}Suppose $(\mu,\lambda) \in \Lambda_{n-1,n+1}$, and
$\mu \ll \lambda$.  Then as $SL_{2}$-modules
\begin{equation*}
V^{\lambda / \mu} \cong \bigotimes_{i=1}^{n}F^{r_{i}(\mu ,\lambda )}
\end{equation*}%
where $SL_{2}$ acts by the tensor product representation on the
right hand side.
\end{proposition}

As a corollary of Theorem \ref{TransferThm2} and Proposition
\ref{Prop1.2} we can describe the $SL_{2}$-module structure of the
multiplicity spaces $W^{\lambda / \mu}$.

\begin{corollary}
\label{CorollaryWY}Suppose $(\mu,\lambda) \in \Lambda_{\B}$.  Then
as $SL_{2}$-modules
\begin{equation*}
W^{\lambda / \mu} \cong \bigotimes_{i=1}^{n}F^{r_{i}(\mu
,\lambda^{+} )},
\end{equation*}%
where $SL_{2}$ acts by the tensor product representation on the right hand side.
\end{corollary}

The above corollary first appeared explicitly as Theorem 3.3,
\cite{WY}, where it was proved using the combinatorics of partition
functions. It can also be obtained using Theorem 5.2, \cite{M2},
where it is shown that $W^{\lambda / \mu}$ carries an irreducible
action of a certain infinite dimensional Hopf algebra called the
twisted Yangian.

\subsection{A resolution of multiplicities}
We now describe the second theorem of this paper. Define
$L=\prod_{i=1}^{n}SL_{2}$.  For $(\mu,\lambda) \in \Lambda_{n-1,n}$
consider the irreducible $L$-module
\begin{equation*}
\label{A} A^{\lambda / \mu} = \bigotimes_{i=1}^{n}F^{r_{i}(\mu
,\lambda^{+} )}.
\end{equation*}
Corollary \ref{CorollaryWY} states that for all $(\mu,\lambda) \in
\Lambda_{\B}$, $W^{\lambda / \mu} \cong Res_{SL_{2}}^{L}A^{\lambda /
\mu}$,
where $SL_{2} \subset L$ is the diagonal subgroup. We therefore ask,
is there a canonical action of $L$ on $W^{\lambda / \mu}$ such that
$W^{\lambda / \mu}\cong A^{\lambda / \mu}$ as $L$-modules?

Remarkably, the answer to this question is yes! We will construct a
canonical action of $L$ on $\B$ that will be uniquely determined by
two properties, the first of which is that each multiplicity space
$W^{\lambda / \mu}$ is isomorphic to $A^{\lambda / \mu}$ as an
$L$-module.  Moreover, the restriction of this action to the
diagonally embedded $SL_{2} \subset L$ recovers the natural action
of $SL_{2}$ on $\B$.

We warn the reader that $L$ is not the product of $SL_{2}$'s that
lives in $Sp_{2n}$. Indeed, the latter product of $SL_{2}$'s does
not act on the multiplicity spaces.  The existence of this
$L$-action is more subtle, and can only be ``seen'' by considering
all multiplicity spaces together, i.e. by considering the branching
algebra.

To describe the action of $L$ on $\B$ precisely we investigate the
double interlacing condition that characterizes branching of the
symplectic groups. Notice that the inequality
\begin{equation*}
\lambda_{i} \geq \mu_{i} \geq \lambda_{i+2}
\end{equation*}
does not constrain the relation between $\mu_{i}$ and $\lambda_{i+1}$.  In other words, we can have either $\mu_{i} \geq \lambda_{i+1}$, or $\mu_{i} \leq \lambda_{i+1}$, or both.  This motivates the following:
\begin{definition}
\label{ordertype}
An \textbf{order type} $\sigma$ is a word in the alphabet $\{\geq,\leq\}$ of length $n-1$.
\end{definition}
Suppose $(\mu,\lambda) \in \Lambda_{\B}$ and $\sigma = (\sigma_{1}
\cdots \sigma_{n-1})$ is an order type.  Then we say $(\mu,\lambda)$
is \textbf{of order type $\sigma$} if for $i=1,...,n-1$,
\begin{equation*}
\left\{
\begin{array}{rl}
\sigma_{i} = \text{``}\geq\text{''} \Longrightarrow \mu_{i} \geq \lambda_{i+1} \\
\sigma_{i} = \text{``}\leq\text{''} \Longrightarrow \mu_{i} \leq \lambda_{i+1}
\end{array} \right.
\end{equation*}
For example, consider the double interlacing pair $(\mu,\lambda)$, where $\lambda = (3,2,1)$ and $\mu = (3,0)$.  Since $\mu_{1}  \geq \lambda_{2}$ and $\mu_{2} \leq \lambda_{3}$, the pair $(\mu, \lambda)$ is of order type $\sigma = (\geq\leq)$.

Let $\Sigma$ be the set of order types, and for each $\sigma \in \Sigma$ set
\begin{equation*}
\Lambda_{\B}(\sigma) = \{(\mu,\lambda) \in \Lambda_{\B} :
(\mu,\lambda) \text{ is of order type } \sigma \}.
\end{equation*}
It's easy to check that $\Lambda_{\B}(\sigma)$ is a sub-semigroup of
$\Lambda_{\B}$.  Therefore
\begin{equation*}
\label{M} \B_{\sigma} = \bigoplus_{(\mu,\lambda) \in
\Lambda_{\B}(\sigma)} W^{\lambda / \mu}
\end{equation*}
is a subalgebra of $\B$.  Moreover, $\B_{\sigma}$ is
$SL_{2}$-invariant.  We now have all the ingredients to state our
second theorem.

\begin{theorem}
\label{ExtThm2}There is a unique representation $(\Phi,\calli{B})$
of $L$ satisfying the following two properties:
\begin{enumerate}
\item for all $(\mu,\lambda) \in \Lambda_{\B}$, $W^{\lambda / \mu}$ is an irreducible $L$-invariant subspace of $\calli{B}$ isomorphic to $A^{\lambda / \mu}$, and
\item for all $\sigma \in \Sigma$, $L$ acts as algebra automorphisms on $\calli{B}_{\sigma}$.
\end{enumerate}
Moreover, $Res_{SL_{2}}^{L}(\Phi)$ recovers the natural action of
$SL_{2}$ on $\calli{B}$.
\end{theorem}

We will now give an overview of the proof of this theorem.  From the statement of Theorem \ref{ExtThm2} it's clear that the
subalgebras $\B_{\sigma}$ are intrinsic to the definition of the
representation $(\Phi,\B)$ of $L$.  Therefore it is not surprising that
the proof requires a thorough understanding of these subalgebras.

To investigate the subalgebras $\B_{\sigma}$ we define the
$L$-module $(\theta_{\sigma},\calli{A}_{\sigma})$, where
\begin{equation}
\label{defAsigma} \mathcal{A}_{\sigma
}=\dbigoplus\limits_{(\mu,\lambda) \in
\Lambda_{\B}(\sigma)}A^{\lambda / \mu}.
\end{equation}
The crucial observation is that since we are restricting to a fixed
order type there is a natural product on $\calli{A}_{\sigma}$.
The  multiplication
\begin{equation*}
A^{\lambda / \mu}\otimes A^{\lambda' / \mu'}\rightarrow A^{\lambda+
\lambda' / \mu + \mu'}
\end{equation*}
is given by Cartan product of irreducible $L$-modules, which induces
an algebra structure on $\mathcal{A}_{\sigma}$ (cf. Lemma
\ref{Lemma5.1}).  The algebra $\mathcal{A}_{\sigma}$ is a naturally
occurring $L$-algebra.  Indeed, in Lemma \ref{OVlemma} we show that
$\mathcal{A}_{\sigma}$ is isomorphic as a graded $L$-algebra to the
ring of polynomials $\mathcal{O}(V)$ on a certain $L$-module $V$.

%

A priori the algebras $\calli{B}_{\sigma}$ and $\calli{A}_{\sigma}$
seem to be quite different.  Indeed, the product maps $A^{\lambda /
\mu}\otimes A^{\lambda' / \mu'}\rightarrow A^{\lambda+ \lambda' /
\mu + \mu'}$ in $\calli{A}_{\sigma}$ are surjective, while the
Cartan product of multiplicity spaces need not be surjective.  For
example, consider the case $\lambda =\lambda ^{\prime }=(2,1,0)$,
$\mu =(2,0)$, and $\mu ^{\prime }=(0,0)$.  By Corollary
\ref{CorollaryWY}, $\dim W^{\lambda / \mu}=\dim W^{\lambda' /
\mu'}=2$, and $\dim W^{\lambda +\lambda ^{\prime } / \mu +\mu
^{\prime }}=9$. \ Therefore the product map cannot be surjective in
this case.

Notice that in the above example $(\mu ,\lambda )$ and $(\mu
^{\prime },\lambda ^{\prime })$ do not satisfy a common order type.
\ An important result for us is that if the multiplicity spaces do
satisfy a common order type, then their product is surjective. This
will be our most crucial application of Theorem \ref{TransferThm2}.
\begin{proposition}
\label{SpnSurjectivityCorollary}Let $\sigma \in \Sigma $ and let
$(\mu ,\lambda ),(\mu ^{\prime },\lambda ^{\prime })\in
\Lambda_{\B}(\sigma)$. \ Then the
map%
\begin{equation*}
W^{\lambda / \mu}\otimes W^{\lambda' / \mu'}\overset{\pi _{\lambda ,\lambda ^{\prime }}}{\rightarrow }%
W^{\lambda+\lambda' / \mu+\mu'}
\end{equation*}%
is surjective.
\end{proposition}

By Proposition \ref{CorollaryWY}, $\calli{B}_{\sigma}$ and
$\calli{A}_{\sigma}$ are isomorphic as $SL_{2}$ modules. The above
proposition shows, moreover, that their products behave similarly.
In fact, we have the following theorem:

\begin{proposition}
\label{CorIsomMtoA} Let $\sigma \in \Sigma$. Then, $ $
\begin{enumerate}
\item There is an isomorphism of graded $SL_{2}$-algebras, $\phi_{\sigma}:\calli{B}_{\sigma} \rightarrow \calli{A}_{\sigma}$, which
is unique up to scalars.
\item By part (1) we can transfer the action of
$L$ on $\calli{A}_{\sigma}$ to $\calli{B}_{\sigma}$, and the
resulting representation, $(\Phi_{\sigma},\B_{\sigma})$, of $L$ is
canonical, i.e. independent of the choice of scalars.
\end{enumerate}
\end{proposition}

We now have a family of $L$-algebras $\{
(\Phi_{\sigma},\B_{\sigma})\}_{\sigma \in \Sigma}$.  By showing that
the action of $L$ is well-defined on the intersection of these
subalgebras, we obtain the representation $(\Phi,\B)$ of $L$ and
prove Theorem \ref{ExtThm2}.

We conclude this section by describing how to use our results to
resolve the multiplicities that occur in branching of symplectic
groups. Recall that generically the branching of the symplectic
groups is not multiplicity free.  It's a fundamental problem in
classical invariant theory to resolve these multiplicities.  Theorem
\ref{ExtThm2} provides a solution to this problem that is rooted in
classical invariant theoretic techniques.  Another solution to this
problem using the theory of quantum groups, in particular Yangians,
appears in \cite{M2}.

Now, it is well known that irreducible $L$-modules have one
dimensional weight spaces. Therefore, by Theorem \ref{ExtThm2} we
obtain a canonical decomposition of the multiplicity spaces
$W^{\lambda / \mu}$ into one-dimensional spaces.  A priori, it seems
that this decomposition depends on a choice of torus of L. We will
show that in fact this choice is induced by the torus of $Sp_{2n}$
that is fixed throughout.  More precisely, we have:
%
%
%
%
%

\begin{corollary}
\label{ExtCor} Let $(\mu,\lambda) \in \Lambda$  There is a canonical
decomposition of $W^{\lambda / \mu}$ into one dimensional spaces
\begin{equation*}
W^{\lambda / \mu}=\bigoplus_{\substack{ \gamma \in \Lambda_{n} \\
\mu < \gamma < \lambda^{+}} } W^{\lambda / \gamma / \mu}.
\end{equation*}%
In particular, $\B$ has a basis which is unique up to scalar.
\end{corollary}

Properties of the basis of $\B$ appearing in Corollary \ref{ExtCor}
are studied in [KY]; in particular we show that it is a standard
monomial basis, i.e. it satisfies a straightening law.  We then use
that to describe an explicit toric deformation of $Spec(\B)$.

\section{Proof of Theorem \ref{TransferThm2}}

\subsection{\label{SectionIndSyst}Some results of Zhelobenko}
Let $G$ be a connected classical group. \ We use freely the notation
from Section \ref{section2.1}.  Let $\lambda \in \Lambda_{G}$.  Then
$F^{\lambda}_{G} \subset \calli{R}_{G}$ embeds linearly in
$\mathcal{O}(U_{G})$ via $res:f{\mapsto }f|_{U_{G}}.$  Set
$Z_{\lambda }(U_{G})=res_{U_{G}}(F^{\lambda }_{G})$.  If there is no
cause for confusion, we write simply $Z_{\lambda }=Z_{\lambda
}(U_{G})$.

We define a representation of $G$ on $Z_{\lambda }$ as follows. \
Let $e^{\lambda }:T_{G}\rightarrow \mathbb{C}$ be the character of
$T_{G}$ given by $t\mapsto t^{\lambda }$. \ We
extend this character to $\overline{U}_{G}T_{G}U_{G}$ by defining $e^{\lambda }(\overline{u}%
tu)=t^{\lambda }$. \ Then by continuity $e^{\lambda }$ is defined on all of $%
G$. \ Now let $u\in U_{G}$, $g\in u^{-1}\overline{U}_{G}T_{G}U_{G}$,
and $f\in Z_{\lambda }$. \ Write $ug=\overline{u}_{1}t_{1}u_{1}\in
\overline{U}_{G}T_{G}U_{G}$. \ Then define
\begin{equation}
\label{Zheloaction} g.f(u)=e^{\lambda }(t_{1})f(u_{1})\text{.}
\end{equation}%
Since $u^{-1}\overline{U}_{G}T_{G}U_{G}$ is dense, we extend this
action to all of $G$. \ Note that the constant
function $z_{\lambda }:u\mapsto 1$ is a canonical highest weight vector in $%
Z_{\lambda }$ of weight $\lambda$, and $res: F_{G}^{\lambda}
\rightarrow Z_{\lambda}$ is an isomorphism of $G$-modules.

Let $\{\alpha _{1},...,\alpha _{n}\}$ be a set of simple roots
relative to the positive roots $\Phi ^{+}$. \ For each $\alpha _{i}$
choose a nonzero root vector $X_{i}\in \mathfrak{g}_{\alpha _{i}}$.
\ Let $D_{i}$ be the differential operator on $\mathcal{O}(U_{G})$
corresponding to the infinitesimal action of $X_{i}$ acting on
$\mathcal{O}(U_{G})$ by left translation. \ Finally, let $\{\varpi
_{1},...,\varpi _{n}\}$ be the fundamental weights and suppose
$\lambda =m_{1}\varpi _{1}+\cdots +m_{n}\varpi _{n}\in \Lambda_{G}$.

\begin{proposition}[{Theorem 1, \S65, Chapter X, \cite{Z}}]
\label{ZheloProp} The space $Z_{\lambda} \subset \calli{O}(U_{G})$
is the solutions to the system of differential equations
$\{D_{i}^{m_{i}+1}=0 : i=1,...,n \}$.   In other words, $$Z_{\lambda
}=\{f\in \mathcal{O}(U_{G}):D_{i}^{m_{i}+1}f=0 \text{ for }
i=1,...,n\}.$$
\end{proposition}

In \cite{Z} the system of differential equations $\{D_{i}^{m_{i}+1}=0:i=1,...,n%
\} $ is termed the ``indicator system''.  Notice that by the Leibniz rule $Z_{\lambda }Z_{\lambda^{\prime }}=Z_{\lambda
+\lambda ^{\prime }}$.

\medskip
Consider the ring $\mathcal{O}(U_{G} \times \mathbb{C}^{n})$. \ Let $%
t_{1},...,t_{n}$ be the standard coordinates on $\mathbb{C}^{n}$. \ Then $%
\mathcal{O}(U_{G} \times
\mathbb{C}^{n})=\tbigoplus\nolimits_{\overline{m}\in
\mathbb{N}^{n}}\mathcal{O}(U_{G})\otimes t_{1}^{m_{1}}\cdots
t_{n}^{m_{n}}$. \ Set $\overline{m}_{\lambda}=(m_{1},...,m_{n})$
where $\lambda =m_{1}\varpi _{1}+\cdots +m_{n}\varpi _{n}$. \ We
form the subring
\begin{equation}
\label{ZheloFlagAlg}
\mathcal{Z}_{G}=\bigoplus_{\lambda \in \Lambda_{G}} Z_{\lambda}\otimes
t^{\overline{m}_{\lambda}}
\end{equation}
of $\mathcal{O}(U_{G} \times \mathbb{C}^{n})$.  This is a $G$-ring,
with $G$ acting on the left factor. Define a map
\begin{equation}
res: \mathcal{R}_{G} \rightarrow \mathcal{Z}_{G}  \label{ZheloIsom}
\end{equation}%
by $$f \in F^{\lambda}_{G} \mapsto f|_{U_{G}} \otimes
t^{\overline{m}_{\lambda}} \in Z_{\lambda}\otimes
t^{\overline{m}_{\lambda}}.$$It's easy to see that:
\begin{proposition}
\label{ZheloRemark}The map $res: \mathcal{R}_{G} \rightarrow \mathcal{Z}_{G}$ is an isomorphism of $G$-rings.
\end{proposition}

\subsection{Preparatory lemmas}
We now specialize the results of the previous section to our
setting.

For $\lambda \in \Lambda_{n}$ we consider Zhelobenko's realization
of $V^{\lambda^{+}}$, which is denoted $Z_{\lambda^{+}}(U_{n+1})$.
So $Z_{\lambda^{+}}(U_{n+1})$ is an irreducible $GL_{n+1}$-module.
Let $x_{ij}$ be the standard coordinates on $U_{n+1}$.

Zhelobenko's realization of the
irreducible $Sp_{2n}$-module of highest weight $\lambda$ is denoted
$Z_{\lambda}(U_{C_{n}})$.  For the affine space $U_{C_{n}}$, the
following can be taken as coordinates:
\begin{equation}
\begin{array}{cccccccc}
1 & u_{12} &  &  & \cdots &  & u_{1n-1} & u_{1n} \\
& 1 & u_{23} &  & \cdots &  & u_{2n-1} &  \\
&  & \ddots & \ddots &  & .^{{\Large .}^{{\Huge .}}} &  &  \\
&  &  & 1 & u_{n,n+1} &  &  &
\end{array}
\label{vars}
\end{equation}%
(The one's are retained here in order to preserve the symmetry of the
entries.) \ The other entries of $U_{C_{n}}$ are polynomials in these coordinates.

The group $U_{n-1}$ acts on $\mathcal{O}(U_{n+1})$ by right
translation, and a straight-forward calculation shows that $f\in
\mathcal{O}(U_{n+1})^{U_{n-1}}$ if, and only if, it's a polynomial
in the variables
\begin{equation*}
\{x_{i,j} : i=1,...,n \text{ }j=n,n+1 \text{ and }i<j\}.
\label{vars2}
\end{equation*}%
Similarly, $\mathcal{O}(U_{C_{n}})^{U_{C_{n-1}}}$ is the polynomial
ring in the variables
\begin{equation*}
\{u_{i,j} : i=1,...,n \text{ }j=n,n+1 \text{ and }i<j\}.
\label{vars3}
\end{equation*}
Hence both $\mathcal{O}(U_{n+1})^{U_{n-1}}$ and $\mathcal{O}(U_{C_{n}})^{U_{C_{n-1}}}$ are polynomial rings in $%
2n-1$ variables.

Let $\psi_{a,b}:M_{m,n} \rightarrow M_{a,b}$ be the map assigning a
matrix its principal $a \times b$ submatrix, and set
$\psi_{a}=\psi_{a,a}$. Let
$\psi_{Z}=\psi_{n+1}|_{U_{C_{n}}}:U_{C_{n}} \rightarrow U_{n+1}$.
The induced map on functions $\psi_{Z}^{*}:\calli{O}(U_{n+1})
\rightarrow \calli{O}(U_{C_{n}})$ satisfies
$\psi_{Z}^{*}(x_{ij})=u_{ij}$.  By our
descriptions of the rings $\mathcal{O}(U_{C_{n}})^{U_{C_{n-1}}}$ and $%
\mathcal{O}(U_{n+1})^{U_{n-1}}$,
$\psi_{Z}^{*}:\mathcal{O}(U_{n+1})^{U_{n-1}} \rightarrow
\mathcal{O}(U_{C_{n}})^{U_{C_{n-1}}} $ is a ring isomorphism.
%
%
%
%
%
%


\begin{lemma}
\label{TransferLemma}Let $\lambda \in \Lambda_{n}$. \ The map
$\psi_{Z}^{*}$ restricts to a linear isomorphism: $$\psi_{Z}^{*}
:Z_{\lambda ^{+}}(U_{n+1})^{U_{n-1}}\overset{\simeq }{\rightarrow
}Z_{\lambda }(U_{C_{n}})^{U_{C_{n-1}}}$$
\end{lemma}

\begin{proof}
Let $\overline{m}_{\lambda}=(m_{1},...,m_{n})$.  In the proof of
Theorem 4, \S 114 in \cite{Z}, Zhelobenko shows that $Z_{\lambda
}(U_{C_{n}})^{U_{C_{n-1}}}$ equals%
\begin{equation*}
\left\{ f\in \mathcal{O}(U_{C_{n}})^{U_{C_{n-1}}}:(u_{i+1,n}\frac{\partial }{%
\partial u_{i,n}}+u_{i+1,n+1}\frac{\partial }{\partial u_{i,n+1}}%
)^{m_{i}+1}(f)=0\text{ for }i=1,...,n\right\} \label{SympIndi}
\end{equation*}
while
$Z_{\lambda ^{+}}(U_{n+1})^{U_{n-1}}$ equals%
\begin{equation*}
\left\{ f\in \mathcal{O}(U_{n+1})^{U_{n-1}}:(x_{i+1,n}\frac{\partial }{%
\partial x_{i,n}}+x_{i+1,n+1}\frac{\partial }{\partial x_{i,n+1}}%
)^{m_{i}+1}(f)=0\text{ for }i=1,...,n\right\} \text{.}  \label{GLIndi}
\end{equation*}%
With these descriptions in hand, and the fact that
$\psi_{Z}^{*}(x_{ij})=u_{ij}$, it follows that $$\psi_{Z}^{*}
(Z_{\lambda ^{+}}(U_{n+1})^{U_{n-1}})=Z_{\lambda
}(U_{C_{n}})^{U_{C_{n-1}}}.$$
\end{proof}
%
%

Now, $T_{C_{n-1}}$ acts on $Z_{\lambda }(U_{C_{n}})^{U_{C_{n-1}}}$, while $%
T_{n-1}$ acts on $Z_{\lambda ^{+}}(U_{n+1})^{U_{n-1}}$. \ In the
next lemma, these tori are both identified with $(\mathbb{C}^{\times
})^{n-1}$.

\begin{lemma}
\label{MainIndiLemma}Let $\lambda \in \Lambda_{n}$. $\ $Then
$\psi_{Z}^{*} :Z_{\lambda ^{+}}(U_{n+1})^{U_{n-1}}\rightarrow
Z_{\lambda }(U_{C_{n}})^{U_{C_{n-1}}}$ is a $(\mathbb{C}^{\times
})^{n-1}$-isomorphism.
\end{lemma}

\begin{proof}
By Lemma \ref{TransferLemma}, it remains to show only that
$\psi_{Z}^{*}$ intertwines the $(\mathbb{C}^{\times
})^{n-1}$-action.  Let $f \in Z_{\lambda
^{+}}(U_{n+1})^{U_{n-1}}$, $t \in (\mathbb{C}^{\times})^{n-1}$,
and $u \in U_{C_{n}}$.  By definition of the action of the tori (cf.
(\ref{Zheloaction})),
$$
\psi_{Z}^{*}(t.f)(u)=e^{\lambda}(t)f(t^{-1}\psi_{Z}(u)t),
$$
while
$$
t.\psi_{Z}^{*}(f)(u)=e^{\lambda}(t)f(\psi_{Z}(t^{-1}ut)).
$$
It's easy to check now that
$\psi_{Z}^{*}(t.f)(u)=t.\psi_{Z}^{*}(f)(u)$.
\end{proof}
%
%
%
%

For $\mu \in \Lambda _{n-1}$ let $Z_{\lambda / \mu }(U_{C_{n}})$
(resp. $Z_{\lambda^{+} / \mu }(U_{n+1})$) denote the $\mu$ weight
spaces of $Z_{\lambda }(U_{C_{n}})^{U_{C_{n-1}}}$ (resp. $Z_{\lambda
^{+}}(U_{n+1})^{U_{n-1}}$).  Lemma \ref{MainIndiLemma} implies
that
$$
\psi_{Z}^{*} :Z_{\lambda^{+} / \mu }(U_{n+1})\overset{\simeq }{%
\rightarrow }Z_{\lambda / \mu }(U_{C_{n}})
$$
is a linear isomorphism. \ These spaces are isomorphic to the
multiplicity spaces $V^{\lambda^{+} / \mu}$ and $W^{\lambda / \mu}$.
Recall that $SL_{2}$ acts on these multiplicity spaces.

\begin{lemma}
\label{OtherIndiLemma}Let $\lambda \in \Lambda_{n}$ and $\mu \in
\Lambda _{n-1}$. \ Then $$ \psi_{Z}^{*} :Z_{\lambda^{+} / \mu
}(U_{n+1}) \rightarrow Z_{\lambda / \mu }(U_{C_{n}})
$$ is an $SL_{2}$-isomorphism.
\end{lemma}

\begin{proof}
As mentioned above, the map $ \psi_{Z}^{*} :Z_{\lambda^{+} / \mu
}(U_{n+1}) \rightarrow Z_{\lambda / \mu }(U_{C_{n}})
$ is a linear isomorphism, so we just need to show it intertwines the $SL_{2}$ action.  Now, the action of of $SL_{2}$ on the multiplicity spaces is defined via its embeddings in $GL_{n+1}$ and $Sp_{2n}$ as
explained in Section \ref{section2.2}.  For the sake of clarity, let us denote these embeddings by
$\alpha:SL_{2} \hookrightarrow GL_{n+1}$ and $\beta:SL_{2} \hookrightarrow Sp_{2n}$.  Note that for $x \in SL_{2}$, we have $\psi_{n+1}(\beta(x))=\alpha(x)$.

Let $f \in Z_{\lambda^{+}}(U_{n+1})^{U_{n-1}}$, $x \in SL_{2}$, and $u \in U_{C_{n}}$.  We want to show that
\begin{equation}
\label{SL2intertwiner}
\psi_{n+1}^{*}(\alpha(x).f)(u)=(\beta(x).(\psi_{n+1}^{*}(f)))(u).
\end{equation}
By definition of the action of $Sp_{2n}$ on $Z_{\lambda}(U_{C_{n}})$
(see (\ref{Zheloaction})), we have
$$
(\beta(x).(\psi_{n+1}^{*}(f)))(u)=e^{\lambda}(t_{1})f(\psi_{n+1}(u_{1})),
$$
where
$$
u\beta(x)=\overline{u}_{1}t_{1}u_{1} \in \overline{U}_{C_{n}}T_{C_{n}}U_{C_{n}}.
$$

To describe the left hand side of (\ref{SL2intertwiner}), we need to first decompose $\psi_{n+1}(u)\alpha(x)$ into
a product compatible with $\overline{U}_{n+1}T_{n+1}U_{n+1}$.  To wit,
\begin{eqnarray*}
\psi_{n+1}(u)\alpha(x) &=& \psi_{n+1}(u)\psi_{n+1}(\beta(x)) \\ &=& \psi_{n+1}(u\beta(x)) \\ &=& \psi_{n+1}(\overline{u}_{1}t_{1}u_{1})
\\ &=& \psi_{n+1}(\overline{u}_{1})\psi_{n+1}(t_{1})\psi_{n+1}(u_{1}).
\end{eqnarray*}
Therefore
$$
\psi_{n+1}^{*}(\alpha(x).f)(u)=e^{\lambda^{+}}(\psi_{n+1}(t_{1}))f(\psi_{n+1}(u_{1})).
$$
Since clearly, $e^{\lambda}(t_{1})=e^{\lambda^{+}}(\psi_{n+1}(t_{1}))$, (\ref{SL2intertwiner}) holds.
\end{proof}

\subsection{\label{section5.5} Proof of Theorem \ref{TransferThm2}}

Recall the homomorphism $\psi^{*}:\calli{O}(M_{n,n+1}) \rightarrow
\calli{O}(Sp_{2n})$ (see (\ref{psistar})), which is induced from the
map $\psi$ taking an element of $Sp_{2n}$ to its principal $n \times
(n+1)$ submatrix.

\begin{lemma}
\label{psistarlemma} We have $\psi^{*}(\B') \subset \B$.  Moreover,
$\psi^{*}:\B' \rightarrow \B$ is an $\Lambda_{n-1,n}$-graded map of
$SL_{2}$ algebras.
\end{lemma}

\begin{proof}
Let $f \in \B', \overline{u} \in \overline{U}_{C_{n}}, g \in
Sp_{2n},$ and $u \in U_{C_{n-1}}$.  We must show that
$f(\psi(\overline{u}gu))=f(\psi(g))$.  Indeed, a straight-forward
computation using block matrices shows that
$$\psi(\overline{u}gu)=\psi_{n}(\overline{u})\psi(g)\psi_{n+1}(u).$$
Since clearly $\psi_{n}(\overline{u}) \in \overline{U}_{n}$ and
$\psi_{n+1}(u) \in U_{n-1}$, the first statement follows.

Now suppose $f \in V^{\lambda^{+} / \mu}$.  Let $t \in T_{C_{n}}$,
$s \in T_{C_{n-1}}$, and $g \in Sp_{2n}$.  Then
\begin{eqnarray*}
(t,s).\psi^{*}(f)(g) &=& \psi^{*}(f)(t^{-1}gs) \\&=&
f(\psi(t^{-1}gs)) \\&=& f(\psi_{n}(t^{-1})\psi(g)\psi_{n-1}(s))
\\&=& \psi_{n}(t)^{-\lambda}\psi_{n-1}(s)^{\mu}f(\psi(g)) \\&=&
t^{-\lambda}s^{\mu}f(\psi(g)) \\&=&
t^{-\lambda}s^{\mu}.\psi^{*}(f)(g).
\end{eqnarray*}
Therefore $\psi^{*}(f) \in W^{\lambda / \mu}$, and hence $\psi^{*}$
is graded.

Finally, we must show $\psi^{*}$ intertwines the $SL_{2}$-action.
Indeed, let $f \in \B'$, $x \in SL_{2}$, and $g \in Sp_{2n}$.  Then
another computation with block matrices shows that
$f(\psi(g)x)=f(\psi(gx))$. (Here one has to be careful to use the
correct embeddings of $SL_{2}$ in $GL_{n+1}$ and $Sp_{2n}$ that
define the corresponding actions.) Therefore,
$x.\psi^{*}(f)=\psi^{*}(x.f)$.
\end{proof}

By the above lemma,$\psi^{*}:\B' \rightarrow \B$ is a morphism of $\Lambda_{n-1,n}$-graded algebras.  To complete the proof of Theoreme \ref{TransferThm2} we must it is an isomorphism.

Let $U_{n,n+1}=\psi_{n,n+1}(U_{n+1})$.  We identify the affine
spaces $U_{n+1}$ and $U_{n,n+1}$ in the obvious way.  For $\lambda
\in \Lambda_{n}$, consider the embedding $Z_{\lambda^{+}}(U_{n+1})
\hookrightarrow \calli{O}(U_{n,n+1})$.  Denote the image of this
embedding by $Z_{\lambda^{+}}(U_{n,n+1})$.  (Similarly, denote by
$Z_{\lambda^{+} / \mu}(U_{n,n+1})$ the image of $Z_{\lambda^{+} /
\mu}(U_{n+1})$ under the embedding.) By transfer of structure,
$Z_{\lambda^{+}}(U_{n,n+1})$ is an irreducible $GL_{n+1}$-module of
highest weight $\lambda^{+}$, i.e. we decree that
$$
\psi_{n,n+1}^{*}:Z_{\lambda^{+}}(U_{n,n+1}) \rightarrow Z_{\lambda^{+}}(U_{n+1})
$$
is an isomorphism of $GL_{n+1}$-modules.  Of course, the inverse of this isomorphism is the map induced by the embedding
$\phi: U_{n,n+1} \rightarrow U_{n+1}$.
Combining this isomorphism with Lemma \ref{OtherIndiLemma} we obtain that
\begin{equation}
\label{SL2isom}
\psi_{n,n+1}^{*}:Z_{\lambda^{+} / \mu}(U_{n,n+1}) \rightarrow Z_{\lambda / \mu}(U_{C_{n}})
\end{equation}
is an isomorphism of $SL_{2}$-modules.

Consider now
$$
\B_{Z}'=\bigoplus_{(\mu,\lambda) \in \Lambda_{n-1,n}} Z_{\lambda^{+} / \mu}(U_{n,n+1}) \otimes t^{\overline{m}_{\lambda^{+}}}
$$
and
$$
\B_{Z}=\bigoplus_{(\mu,\lambda) \in \Lambda_{n-1,n}} Z_{\lambda / \mu}(U_{C_{n}}) \otimes t^{\overline{m}_{\lambda}}.
$$
These are subalgebras of $\mathcal{Z}_{GL_{n+1}}$ and $\mathcal{Z}_{Sp_{2n}}$, respectively (see (\ref{ZheloFlagAlg})).

We define a map
$
 \B_{Z}' \rightarrow \B_{Z}
$
which on graded components is simply given by
$$
\psi_{n,n+1}^{*}\otimes 1:Z_{\lambda^{+} / \mu}(U_{n,n+1}) \otimes t^{\overline{m}_{\lambda^{+}}} \rightarrow Z_{\lambda / \mu}(U_{C_{n}}) \otimes t^{\overline{m}_{\lambda}}.
$$
Let us denote the total map by $\psi_{n,n+1}^{*}\otimes 1 : \B_{Z}' \rightarrow \B_{Z}$ also.

By (\ref{SL2isom}) $\psi_{n,n+1}^{*}\otimes 1 : \B_{Z}' \rightarrow \B_{Z}$ is a isomorphism of $SL_{2}$-modules.  But clearly the map is a morphism of graded algebras, so in fact it is an isomorphism of $\Lambda_{n-1,n}$-graded $SL_{2}$-algebras.  Now, by Proposition \ref{ZheloRemark}, the restriction of $res : \mathcal{R}_{GL_{n+1}} \rightarrow \mathcal{Z}_{GL_{n+1}}$ to $\B'$ gives an isomorphism
$res: \B' \rightarrow \B_{Z}'$ of $\Lambda_{n-1,n}$-graded $SL_{2}$-algebras.  Similarly, we have the isomorphism $res: \B \rightarrow \B_{Z}$ of $\Lambda_{n-1,n}$-graded $SL_{2}$-algebras.

We now have a diagram:
\begin{eqnarray*}{}
  \B' & \longrightarrow & \B \\
  \downarrow &  & \downarrow \\
  \B_{Z}' & \longrightarrow & \B_{Z}
\end{eqnarray*}
where the vertical arrows are given by $res$, the bottom arrow is $\psi_{n,n+1}^{*}\otimes 1$, and the top arrow is $\psi^{*}$.  Clearly, this diagram commutes.  Indeed, this follows from the simple fact that for $f \in \calli{O}(M_{n,n+1})$,
$$
(f \circ \psi)|_{U_{C_{n}}}=f|_{U_{n,n+1}}\circ \psi_{n,n+1}
$$
as elements of $\calli{O}(U_{C_{n}})$.
Moreover, by the previous paragraph the bottom three maps are isomorphisms of $\Lambda_{n-1,n}$-graded $SL_{2}$-algebras by.  We conclude that $\psi^{*}:\B' \rightarrow \B$ is an isomorphism of $\Lambda_{n-1,n}$-graded $SL_{2}$-algebras, thus completing the proof of Theorem \ref{TransferThm2}.


\section{Proof of Proposition \ref{SpnSurjectivityCorollary} \label{SectionDIP}}
In this section we prove various structural results about the
multiplicity spaces $V_{\lambda / \mu}$ where $(\mu,\lambda) \in
\Lambda_{n-1,n+1}$.  Notice that these are multiplicity spaces that
occur in branching from $GL_{n+1}$ to $GL_{n-1}$.  By virtue of
Theorem \ref{TransferThm2}, these results have analogous in the
setting of branching of symplectic groups, and it is for this reason
that these results are important for us.

We will work in slightly greater generality than strictly necessary, and consider the semigroup
\begin{eqnarray*}
\Omega &=& \{ (\mu,\lambda) \in \Lambda_{n-1,n+1} : V_{\lambda /
\mu} \neq \{0\} \} \\ &=& \{ (\mu,\lambda) \in \Lambda_{n-1,n+1} :
\mu \ll \lambda \}.
\end{eqnarray*}
The second equality follows from (\ref{Eq1.2}).

\subsection{\label{section4.0}The rearrangement function}
We begin by introducing the \textbf{rearrangement function} on $\Omega$. $\ $%
Define $f:\Omega\rightarrow \Lambda_{2n}$ by:
\begin{equation*}
(\mu ,\lambda )\overset{f}{\mapsto }(x_{1},y_{1},...,x_{n},y_{n})
\end{equation*}%
where $\{x_{1}\geq y_{1}\geq \cdots \geq x_{n}\geq y_{n}\}$ is the
non-increasing rearrangement of $ (\mu _{1},...,\mu _{n-1},\lambda
_{1},...,\lambda _{n+1}). $
Notice that $f(\mu ,\lambda )$ equals%
\begin{equation*}
(\lambda _{1}\geq \max (\mu _{1},\lambda _{2})\geq \min (\mu _{1},\lambda
_{2})\geq \cdots \geq \max (\mu _{n-1},\lambda _{n})\geq \min (\mu
_{n-1},\lambda _{n})\geq \lambda _{n+1})\text{.}
\end{equation*}%
This easily implies:
\begin{lemma}
\label{simplelemma}Let $(\mu ,\lambda )\in \Omega$. \ Suppose $f(\mu
,\lambda )=(x_{1},y_{1},...,x_{n},y_{n})$ and $\gamma \in \Lambda_{n}$.
\ Then $\mu <\gamma <\lambda $ if, and only if, $y_{i}\leq \gamma _{i}\leq
x_{i}$ for $i=1,...,n$, where $\gamma=(\gamma_{1},...,\gamma_{n})$.
\end{lemma}
For $\sigma \in \Sigma $ let $\Omega(\sigma)$ be the sub-semigroup of $\Omega$ consisting of the pairs of order type $\sigma$.  Let $f_{\sigma }$ denote the restriction of $f$ to $\Omega(\sigma)$.

\begin{lemma}
\label{semigroupisom}Let $\sigma \in \Sigma $. Then $f_{\sigma }:\Omega(\sigma)\rightarrow \Lambda_{2n}$
is a semigroup isomorphism.
\end{lemma}

\begin{proof}
For $(\mu ,\lambda )\in \Omega(\sigma)$ let $f_{\sigma }(\mu
,\lambda )=(f_{\sigma }(\mu ,\lambda )_{1},...,f_{\sigma }(\mu ,\lambda
)_{2n})$. \ Define functions
\begin{eqnarray*}
a:\{1,...,n-1\}\rightarrow\{1,...,2n\} \\
b:\{1,...,n+1\}\rightarrow\{1,...,2n\}
\end{eqnarray*}
by $b(1)=1$ and $b(n+1)=2n$, and for $i=1,...,n-1$
\begin{eqnarray*}
(\sigma _{i}\text{ is} &\geq &)\Longrightarrow a(i)=2i\text{ and }b(i+1)=2i+1
\\
(\sigma _{i}\text{ is} &\leq &)\Longrightarrow a(i)=2i+1\text{ and }b(i+1)=2i.
\end{eqnarray*}%
Then for all $%
(\mu ,\lambda )\in \Omega(\sigma)$%
\begin{eqnarray*}
\mu _{i} &=&f_{\sigma }(\mu ,\lambda )_{a(i)}\text{ for }i=1,...,n-1 \\
\lambda _{j} &=&f_{\sigma }(\mu ,\lambda )_{b(j)}\text{ for }j=1,...,n+1.
\end{eqnarray*}
This implies that $f_{\sigma }$ is an injective semigroup homomorphism.

Now suppose $(z_{1},...,z_{2n})\in \Lambda_{2n}$ is given. \ Define $\mu $ and $\lambda $ by the formulas
\begin{eqnarray*}
\mu _{i} &=&z_{a(i)}\text{ for }i=1,...,n-1 \\
\lambda _{j} &=&z_{b(j)}\text{ for }j=1,...,n+1\text{.}
\end{eqnarray*}%
Since $a(1)<a(2)<\cdots <a(n-1)$, it follows that $\mu \in \Lambda_{n-1}$. \ Similarly, $%
\lambda \in \Lambda _{n+1}^{+}$. \ Since $b(i)<a(i)<b(i+2)$, we get that $%
\mu \ll \lambda $. \ Finally, suppose $\sigma _{i}$ is $\geq $. \ Then $%
a(i)<b(i+1)$, and so $\mu _{i}\geq \lambda _{i+1}$. \ Similarly, if $\sigma
_{i}$ is $\leq $ then $\mu _{i}\leq \lambda _{i+1}$. \ Therefore $(\mu
,\lambda )\in \Omega(\sigma)$. \ Since $f_{\sigma }(\mu ,\lambda
)=(z_{1},...,z_{2n})$ we conclude that $f_{\sigma }$ is surjective.
\end{proof}

\begin{lemma}
\label{peelinglemma}Let $\sigma \in \Sigma $ and let $(\mu ,\lambda ),(\mu
^{\prime },\lambda ^{\prime })\in \Omega(\sigma)$. \ Suppose that $%
\gamma \in \Lambda_{n}$ satisfies%
\begin{equation*}
\mu +\mu ^{\prime }<\gamma <\lambda +\lambda ^{\prime }\text{.}
\end{equation*}%
Then there exist $\nu $,$\nu ^{\prime }\in \Lambda_{n}$ such that $%
\gamma =\nu +\nu ^{\prime }$, $\mu <\nu <\lambda $, and $\mu ^{\prime }<\nu
^{\prime }<\lambda ^{\prime }$.
\end{lemma}

\begin{proof}
Set $f_{\sigma }(\mu ,\lambda )=(x_{1},y_{1},...,x_{n},y_{n})$ and $%
f_{\sigma }(\mu ^{\prime },\lambda ^{\prime })=(x_{1}^{\prime
},y_{1}^{\prime },...,x_{n}^{\prime },y_{n}^{\prime })$. \ By Lemma \ref%
{semigroupisom},
\begin{equation*}
f_{\sigma }(\mu +\mu ^{\prime },\lambda +\lambda ^{\prime
})=(x_{1}+x_{1}^{\prime },y_{1}+y_{1}^{\prime },...,x_{n}+x_{n}^{\prime
},y_{n}+y_{n}^{\prime })\text{.}
\end{equation*}%
Therefore by Lemma \ref{simplelemma}, $y_{i}+y_{i}^{\prime }\leq \gamma
_{i}\leq x_{i}+x_{i}^{\prime }$. \ Now choose $\nu _{i},\nu _{i}^{\prime }$ such that $\gamma _{i}=\nu _{i}+\nu
_{i}^{\prime }$, $y_{i}\leq \nu _{i}\leq x_{i}$, and $y_{i}^{\prime }\leq
\nu _{i}^{\prime }\leq x_{i}^{\prime }$.  \ Set $\nu =(\nu _{1},...,\nu _{n})$
and $\nu ^{\prime }=(\nu _{1}^{\prime },...,\nu _{n}^{\prime })$. \ Clearly $%
\nu ,\nu ^{\prime }\in \Lambda_{n}$ and $\gamma =\nu +\nu ^{\prime }$.
\ Moreover, by Lemma \ref{simplelemma}, $\mu <\nu <\lambda $, and $\mu
^{\prime }<\nu ^{\prime }<\lambda ^{\prime }$.
\end{proof}

\subsection{\label{section4.1}Proof of Proposition \ref{Prop1.2}}
For $(\mu,\lambda )\in \Omega$ recall that $r_{i}(\mu ,\lambda
)=x_{i}-y_{i}$, where $f(\mu ,\lambda
)=(x_{1},y_{1},...,x_{n},y_{n})$.  Let $(\mu,\lambda) \in \Lambda
_{n-1,n+1}$.  Then as a $GL_{1}\times GL_{1}$-module,
\begin{equation*}
V^{\lambda / \mu}\cong \bigoplus_{\substack{ \gamma \in \Lambda_{n}
\\ \mu <\gamma <\lambda }}V^{\gamma / \mu}\otimes V^{\lambda / \gamma}%
\text{.}
\end{equation*}%
Now, $d(q)=
\left( \begin{array}{cc}
q & 0  \\
0 & q^{-1} \end{array} \right)$ acts on $V^{\gamma / \mu}\otimes
V^{\lambda / \gamma}$ by the scalar $q^{2\left\vert \gamma
\right\vert -\left\vert \lambda \right\vert -\left\vert \mu
\right\vert }$. \ Moreover, by (\ref{Eq1.1}) and (\ref{Eq1.2}),
$\dim V^{\gamma / \mu}\otimes V^{\lambda / \gamma}=1$ for $\gamma $
such that $\mu <\gamma <\lambda $. \ Therefore the character equals
\begin{equation*}
ch(V^{\lambda / \mu})=\sum_{\substack{ \gamma \in \Lambda_{n}  \\
\mu <\gamma <\lambda }}q^{2\left\vert \gamma \right\vert -\left\vert
\lambda \right\vert -\left\vert \mu \right\vert }\text{.}
\end{equation*}%
Set $f(\mu ,\lambda )=(x_{1},y_{1},...,x_{n},y_{n})$ and $%
r_{i}=r_{i}(\mu ,\lambda )$.\ \ By Lemma \ref{simplelemma},%
\begin{eqnarray*}
\sum_{\substack{ \gamma \in \Lambda_{n}  \\ \mu <\gamma <\lambda }}%
q^{2\left\vert \gamma \right\vert -\left\vert \lambda \right\vert
-\left\vert \mu \right\vert } &=&\sum_{0\leq j_{i}\leq r_{i}}q^{2(y_{1}+\cdots +y_{n}+j_{1}+\cdots
+j_{n})-(x_{1}+\cdots +x_{n}+y_{1}+\cdots +y_{n})} \\
&=&\sum_{0\leq j_{i}\leq r_{i}}q^{(-r_{1}+2j_{1})+\cdots +(-r_{n}+2j_{n})} \\
&=&\prod_{i=1}^{n}\sum_{j=0}^{r_{i}}q^{-r_{i}+2j} \\
&=&\prod_{i=1}^{n}ch(F^{r_{i}})
\end{eqnarray*}%
Therefore $ch(V^{\lambda / \mu})=ch\left(
\bigotimes_{i=1}^{n}F^{r_{i}(\mu ,\lambda )}\right) $.  This proves
Proposition \ref{Prop1.2}.

\subsection{\label{section4.2}A technical lemma}

Suppose $(\gamma,\lambda) \in \Lambda _{n,n+1}$. \ We may view
$V^{\lambda }$ as a $GL_{n}$-module by
restriction, and, as such, define $V^{\lambda }[\gamma ]$ to be the $\gamma $%
-isotypic component of $V^{\lambda }$. \ Let $p_{\gamma }^{\lambda
}:V^{\lambda }\rightarrow V^{\lambda }[\gamma ]$ be the
corresponding projection.

Before stating and proving the lemma we first make a simple observation.  Suppose $\Theta$ is a semigroup and $\mathcal{V}$,$\mathcal{W}$ are $%
\Theta$-graded vector spaces:%
\begin{equation*}
\mathcal{V}=\bigoplus_{i\in \Theta}V_{i} \text{, } \mathcal{W}=\bigoplus_{i\in \Theta}W_{i}.
\end{equation*}%
Suppose there are linear maps $\pi _{i,j}:V_{i}\otimes V_{j}\rightarrow
V_{i+j}$ and $\tau _{i,j}:W_{i}\otimes W_{j}\rightarrow W_{i+j}$ for every $%
i,j\in \Theta$. \ We refer to these maps as ``products'' on the
vector spaces. \ Finally, suppose also there is an $\Theta$-graded
isomorphism $T:\mathcal{V}\rightarrow \mathcal{W}$ that preserves the products on $\mathcal{V}$ and $\mathcal{W}$ in
the following sense: for all $i,j\in \Theta$ the following diagram
commutes:
\begin{equation*}
\begin{array}{ccc}
V_{i}\otimes V_{j} & \overset{\pi _{i,j}}{\longrightarrow } & V_{i+j} \\
\downarrow T &  & \downarrow T \\
W_{i}\otimes W_{j} & \overset{\tau _{i,j}}{\longrightarrow } & W_{i+j}%
\end{array}%
\end{equation*}%
Then if $x\in V_{i}$ and $y\in V_{j}$ and $\tau _{i,j}(T(x)\otimes
T(y))\not=0$, then $\pi _{i,j}(x\otimes y)\not=0$. \

For the purposes of the following lemma we will use the branching semigroup
\begin{equation*}
\Theta=\{(\gamma ,\lambda )\in
\Lambda_{n,n+1}:\gamma <\lambda \}.
\end{equation*}
We introduce three $\Theta$-graded vector spaces, each of which is equipped
with product maps.

The first space is $\mathcal{V}_{1}=\dbigoplus\limits_{(\gamma ,\lambda )\in \Theta%
}V^{\lambda }[\gamma ]$.  The product is defined as follows: for
$x\in V^{\lambda }[\gamma ]$ and $x^{\prime }\in V^{\lambda ^{\prime
}}[\gamma ^{\prime }]$, define $$xx^{\prime }=p_{\nu +\nu ^{\prime
}}^{\lambda +\lambda ^{\prime }}(\pi _{\lambda ,\lambda ^{\prime
}}(x\otimes x^{\prime }))\text{.}$$

The second space is $\mathcal{V}_{2}=\dbigoplus\limits_{(\gamma
,\lambda )\in \Theta}V^{\gamma }\otimes Hom_{GL_{n}}(V^{\gamma
},V^{\lambda })$.  The product is defined as follows: for $v\otimes
f\in V^{\gamma }\otimes Hom_{GL_{n}}(V^{\gamma },V^{\lambda })$ and
$v^{\prime }\otimes f^{\prime }\in V^{\gamma ^{\prime }}\otimes
Hom_{GL_{n}}(V^{\gamma ^{\prime }},V^{\lambda ^{\prime }})$, define
\begin{equation*}
(v\otimes f)(v^{\prime }\otimes f^{\prime })=\pi _{\gamma ,\gamma ^{\prime
}}(v\otimes v^{\prime })\otimes (\pi _{\lambda ,\lambda ^{\prime }}\circ
(f\otimes f^{\prime })\circ j_{\gamma ,\gamma ^{\prime }})\text{.}
\end{equation*}%

The third space is $\mathcal{V}_{3}=\dbigoplus\limits_{(\gamma
,\lambda )\in \Theta}V^{\gamma }\otimes (V^{\lambda
})^{U_{n}}(\gamma )$.  The product is defined as follows: for
$v\otimes w\in V^{\gamma }\otimes (V^{\lambda })^{U_{n}}(\gamma )$
and $v^{\prime }\otimes w^{\prime }\in V^{\gamma ^{\prime }}\otimes
(V^{\lambda ^{\prime }})^{U_{n}}(\gamma ^{\prime })$, define
\begin{equation*}
(v\otimes w)(v^{\prime }\otimes w^{\prime })=\pi _{\gamma ,\gamma ^{\prime
}}(v\otimes v^{\prime })\otimes \pi _{\lambda ,\lambda ^{\prime }}(w\otimes
w^{\prime })\text{.}
\end{equation*}

Finally, we can state and prove the lemma.  We note that it can also be obtained as a consequence of Theorem 1 in \cite{Vin}.  We include an elementary proof for the sake of completeness.

\begin{lemma}
\label{projection lemma}Let $(\nu,\lambda),(\nu',\lambda') \in \Lambda_{n,n+1}$. \ Suppose that $%
0\not=x\in V^{\lambda }[\nu ]$ and $0\not=x^{\prime }\in V^{\lambda
^{\prime }}[\nu ^{\prime }]$. \ Then $p_{\nu +\nu ^{\prime
}}^{\lambda +\lambda ^{\prime }}(\pi _{\lambda ,\lambda ^{\prime
}}(x\otimes x^{\prime }))\not=0$.
\end{lemma}

\begin{proof}
Define $T:\mathcal{V}_{2}\rightarrow \mathcal{V}_{1}$ by $T(v\otimes f)=f(v)$. \ This is
clearly a linear isomorphism. \ Let $v\otimes f$,$v^{\prime }\otimes
f^{\prime }$ be chosen as in the definition of $\mathcal{V}_{2}$. \ Then
\begin{equation*}
T((v\otimes f)(v^{\prime }\otimes f^{\prime }))=(\pi _{\lambda ,\lambda
^{\prime }}\circ (f\otimes f^{\prime })\circ j_{\gamma ,\gamma ^{\prime
}})(\pi _{\gamma ,\gamma ^{\prime }}(v\otimes v^{\prime }))
\end{equation*}%
while
\begin{equation*}
T(v\otimes f)T(v^{\prime }\otimes f^{\prime })=p_{\gamma +\gamma ^{\prime
}}^{\lambda +\lambda ^{\prime }}(\pi _{\lambda ,\lambda ^{\prime
}}(f(v)\otimes f(v^{\prime }))).
\end{equation*}%
Let $z=v\otimes v^{\prime }$. \ Since $%
z\in V^{\gamma }\otimes V^{\gamma ^{\prime }}$ we can write
$z=z_{0}+z_{1}$, where $z_{0}\in (V^{\gamma }\otimes V^{\gamma
^{\prime }})[\gamma +\gamma ^{\prime }]$ and $z_{1}\in \sum_{\tau
\not=\gamma +\gamma ^{\prime }}(V^{\gamma }\otimes V^{\gamma
^{\prime }})[\tau ]$. \ By the definition of $\pi _{\gamma ,\gamma
^{\prime }}$ and $j_{\gamma ,\gamma ^{\prime }}$, the composition
$j_{\gamma ,\gamma ^{\prime }}\circ \pi _{\gamma ,\gamma ^{\prime
}}$ is the projection of $V^{\gamma }\otimes V^{\gamma ^{\prime }}$
onto its isotypic component $(V^{\gamma }\otimes V^{\gamma ^{\prime
}})[\gamma +\gamma ^{\prime }]$. \ Therefore,%
\begin{equation*}
(\pi _{\lambda ,\lambda ^{\prime }}\circ (f\otimes f^{\prime })\circ
j_{\gamma ,\gamma ^{\prime }})(\pi _{\gamma ,\gamma ^{\prime }}(z))=(\pi
_{\lambda ,\lambda ^{\prime }}\circ (f\otimes f^{\prime }))(z_{0})\text{.}
\end{equation*}%
On the other hand,
\begin{eqnarray*}
p_{\gamma +\gamma ^{\prime }}^{\lambda +\lambda ^{\prime }}(\pi _{\lambda
,\lambda ^{\prime }}(f(v)\otimes f(v^{\prime }))) &=&p_{\gamma +\gamma ^{\prime }}^{\lambda +\lambda ^{\prime }}(\pi _{\lambda
,\lambda ^{\prime }}((f\otimes f^{\prime })(z_{0})))+p_{\gamma +\gamma
^{\prime }}^{\lambda +\lambda ^{\prime }}(\pi _{\lambda ,\lambda ^{\prime
}}((f\otimes f^{\prime })(z_{1}))) \\
&=&\pi _{\lambda ,\lambda ^{\prime }}((f\otimes f^{\prime })(z_{0}))\text{.}
\end{eqnarray*}%
Therefore $T$ preserves the products on $\mathcal{V}_{1}$ and $\mathcal{V}_{2}$.

Next define $S:\mathcal{V}_{2}\rightarrow \mathcal{V}_{3}$ by $%
S(v\otimes f)=v\otimes f(v_{\gamma })$ (recall that $v_{\gamma }$ is
the canonical highest weight vector in $V^{\gamma }$). \ This is
clearly a linear isomorphism. \ We show $S$ preserves the product
maps. \ Let $v\otimes f$,$v^{\prime }\otimes f^{\prime }$ be chosen
as in the definition of $\mathcal{V}_{2}$. \ Then
\begin{eqnarray*}
S((v\otimes f)(v^{\prime }\otimes f^{\prime }))
&=&S(\pi _{\gamma ,\gamma
^{\prime }}(v\otimes v^{\prime })\otimes (\pi _{\lambda ,\lambda ^{\prime
}}\circ (f\otimes f^{\prime })\circ j_{\gamma ,\gamma ^{\prime }})) \\
&=&\pi_{\gamma ,\gamma ^{\prime }}(v\otimes v^{\prime })\otimes \pi _{\lambda
,\lambda ^{\prime }}(f(v_{\gamma })\otimes f(v_{\gamma ^{\prime }})) \\
&=& S(v\otimes f)S(v^{\prime }\otimes f^{\prime }).
\end{eqnarray*}%
Therefore $S$ preserves the products on $\mathcal{V}_{2}$ and $\mathcal{V}_{3}$.

Now $S\circ T^{-1}$ is a graded isomorphism of $\mathcal{V}_{1}$ and
$\mathcal{V}_{3}$ that respects products. \ Consider $0\not=x\in
V^{\lambda }[\nu ]$. \ Under the isomorphism $S\circ T^{-1}$, $x$ is
mapped to a simple tensor $v\otimes w$. \ Indeed, by (\ref{Eq1.1})
$\dim (V^{\lambda })^{U_{n}}(\nu )=1$, and $x$ is mapped to the
summand $V^{\nu
}\otimes (V^{\lambda })^{U_{n}}(\nu )$. \ Similarly, $%
0\not=x^{\prime }\in V^{\lambda ^{\prime }}[\nu ^{\prime }]$ is
mapped to a simple tensor $v^{\prime }\otimes w^{\prime }$. \ By the
definition of multiplication in $\calli{V}_{3}$, we have that
$(v\otimes w)(v^{\prime }\otimes w^{\prime })=\pi _{\nu ,\nu
^{\prime }}(v\otimes v^{\prime })\otimes \pi _{\lambda ,\lambda
^{\prime }}(w\otimes w^{\prime })$. \ Now, $\pi _{\nu ,\nu ^{\prime
}}(v\otimes v^{\prime })$ (resp. $\pi _{\lambda ,\lambda
^{\prime }}(w\otimes w^{\prime })$) is simply the product of $v$ and $%
v^{\prime }$ (resp. $w$ and $w^{\prime }$) in $\mathcal{R}_{GL_{n}}$ (resp. $%
\mathcal{R}_{GL_{n+1}}$). \ Since $\mathcal{R}_{GL_{n}}$ (resp. $\mathcal{R}_{GL_{n+1}}$) has no zero divisors, it follows that $%
(v\otimes w)(v^{\prime }\otimes w^{\prime })\not=0$. \ By the observation
above we conclude that $xx^{\prime }\not=0$ in $\mathcal{V}_{1}$, i.e. $p_{\gamma
+\gamma ^{\prime }}^{\lambda +\lambda ^{\prime }}(\pi _{\lambda ,\lambda
^{\prime }}(x\otimes x^{\prime }))\not=0$.
\end{proof}

\subsection{Proof of Proposition \ref{SpnSurjectivityCorollary}}
Let $(%
\mathfrak{t}_{n}^{\ast })_{\mathbb{R}}$ be the real form of $\mathfrak{t}%
_{n}^{\ast }$ spanned by $\{\varepsilon _{i}:i=1,...,n\}$, where
$\varepsilon_{i}$ is the functional mapping a diagonal matrix to its
$i^{th}$ entry. \ Let $(\cdot ,\cdot )$ be the inner product on
$(\mathfrak{t}_{n}^{\ast })_{\mathbb{R}}$ defined by $(\varepsilon
_{i},\varepsilon _{j})=\delta _{ij}$, and let $\left\Vert \gamma
\right\Vert ^{2}=(\gamma ,\gamma )$ define the associated norm. \
Denote by $\preceq $ the positive root ordering on
$\mathfrak{t}_{n}^{\ast }$, defined relative to the set of positive
roots: $\{\varepsilon _{i}-\varepsilon _{j}:i<j\}$. \ In other
words, $\alpha \preceq \beta $ means $\beta -\alpha $ is a
nonnegative integer combination of positive roots. \ Recall that for
$\nu ,\nu ^{\prime },\gamma \in \Lambda_{n}$,
$Hom_{GL_{n}}(V^{\gamma },V^{\nu }\otimes V^{\nu ^{\prime
}})\not=\{0\}$ implies $\gamma \preceq \nu +\nu ^{\prime }$.

By Theorem \ref{TransferThm2}, Proposition \ref{SpnSurjectivityCorollary} is an immediate corollary of the following result:

\begin{proposition}
\label{SurjectivityLemmaGL}Let $\sigma \in \Sigma $ and let $(\mu ,\lambda
),(\mu ^{\prime },\lambda ^{\prime })\in \Omega(\sigma)$. \ Then the
map%
\begin{equation*}
V^{\lambda / \mu}\otimes V^{\lambda' / \mu'}\overset{%
\pi _{\lambda ,\lambda ^{\prime }}}{\rightarrow }V^{\lambda+\lambda'
/ \mu+\mu'}
\end{equation*}%
is surjective.
\end{proposition}
\begin{proof}
To ease notation let $X=V^{\lambda / \mu}$, $X^{\prime }=V^{\lambda'
/ \mu'}$, $Y=V^{\lambda+\lambda' / \mu+\mu'}$, and $\pi =\pi
_{\lambda ,\lambda ^{\prime }}$. \ For $\gamma \in \Lambda_{n}$ set
$Y[\gamma ]=p_{\gamma }^{\lambda +\lambda ^{\prime }}(Y)$, $X[\gamma
]=p_{\gamma }^{\lambda }(X)$, and $X^{\prime }[\gamma ]=p_{\gamma
}^{\lambda ^{\prime }}(X^{\prime })$.

Note that $Y=\bigoplus_{\gamma \in \Lambda_{n}}Y[\gamma ]$, and $\dim Y[\gamma ]$ is zero or one. \
Moreover, $Y[\gamma ]\not=\{0\}$ if, and only if, $\mu +\mu ^{\prime
}<\gamma <\lambda +\lambda ^{\prime }$. \ We will prove by induction on $%
\left\Vert \gamma \right\Vert $ that $Y[\gamma ]$ is in the image of $\pi $.

Let $\gamma \in \Lambda_{n}$ be of minimal norm such that $Y[\gamma
]\not=\{0\}$. \ Our base case is to show that $Y[\gamma ]$ is in the image
of $\pi $. \ Since $(\mu ,\lambda ),(\mu ^{\prime },\lambda ^{\prime })\in
\Omega(\sigma)$ we can apply Lemma \ref{peelinglemma} to obtain $\nu
$,$\nu ^{\prime }\in \Lambda_{n}$ such that $\gamma =\nu +\nu ^{\prime
} $, $\mu <\nu <\lambda $, and $\mu ^{\prime }<\nu ^{\prime }<\lambda
^{\prime }$. \ Choose $0\not=x\in X[\nu ]$ and $0\not=x^{\prime }\in
X^{\prime }[\nu ^{\prime }]$, and let $z=x\otimes x^{\prime }$.

Now, $\pi (z)=\sum_{\tau \in \Lambda_{n}}p_{\tau }^{\lambda +\lambda
^{\prime }}(\pi (z))$ is a decomposition of $\pi (z)$ in $%
Y=\bigoplus_{\gamma \in \Lambda_{n}}Y[\gamma ]$.\ \ Since $p_{\tau }^{\lambda +\lambda ^{\prime }}(\pi (z))=0$ for $%
\tau \succ \gamma $, $\pi (z)=\sum_{\tau \preceq \gamma }p_{\tau
}^{\lambda +\lambda ^{\prime }}(\pi (z))$. \ Now $\tau \prec \gamma $
implies $\left\Vert \tau \right\Vert <\left\Vert \gamma \right\Vert $, and
by hypothesis $\gamma $ is of minimal norm such that $Y[\gamma ]\not=\{0\}$.  Therefore
$p_{\tau }^{\lambda +\lambda ^{\prime }}(\pi (z))=0$ for $\tau \prec \gamma $, and hence $\pi (z)=p_{\gamma }^{\lambda +\lambda ^{\prime }}(\pi (z))\in
Y[\gamma ]$. \ By definition, $\pi (z)$ is the product of $x$ and $x^{\prime
}$ in $\mathcal{R}_{GL_{n+1}}$. \ Therefore, since $\mathcal{R}_{GL_{n+1}}$\ has no zero divisors, $\pi (z)\not=0$. \ Since $\dim Y[\gamma ]=1$, we conclude that
$Y[\gamma ]$ is in the image of $\pi $. \ This completes the base case.

Now fix $\gamma \in \Lambda_{n}$ such that $Y[\gamma ]\not=\{0\}$, and
suppose $Y[\tau ]$ is in the image of $\pi $ for all $\tau $ such that $%
\left\Vert \tau \right\Vert <\left\Vert \gamma \right\Vert $. \ Using Lemma %
\ref{peelinglemma} again, we choose $\nu $,$\nu ^{\prime }\in \Lambda_{n-1}$ such that $\gamma =\nu +\nu ^{\prime }$, $\mu <\nu <\lambda $,
and $\mu ^{\prime }<\nu ^{\prime }<\lambda ^{\prime }$. \ Also choose $%
0\not=y\in Y[\gamma ]$, $0\not=x\in X[\nu ]$, and $0\not=x^{\prime }\in
X^{\prime }[\nu ^{\prime }]$. \ By Lemma \ref{projection lemma}, $p_{\gamma
}^{\lambda +\lambda ^{\prime }}(\pi (x\otimes x^{\prime }))\not=0$. \
Therefore
\begin{equation*}
\pi (x\otimes x^{\prime })\in \mathbb{C}^{\times }y+\dsum\limits_{\tau \prec
\gamma }Y[\tau ]\text{.}
\end{equation*}%
Since $\tau \prec \gamma $ implies $\left\Vert \tau \right\Vert <\left\Vert
\gamma \right\Vert $, by the inductive hypothesis we obtain an element $\xi
\in X\otimes X^{\prime }$ such that $\pi (\xi )=y$. \ Since $\dim Y[\gamma
]=1$, this shows that $Y[\gamma ]$ is in the image of $\pi $. \ This
completes the induction.
\end{proof}

\section{Proof of Theorem \ref{ExtThm2}}
In this section it will be convenient for us to introduce the
following convention.  Elements of the branching semigroup
$\Lambda_{\B}$ will be thought of as ``skew shapes'', and so instead
of writing $(\mu,\lambda) \in \Lambda_{\B}$, we will write $\lambda
/ \mu \in \Lambda_{\B}$.  In this way, for $p=\lambda / \mu \in
\Lambda_{\B}$ we associate the spaces $W^{p}$, $A^{p}$, etc...

\subsection{\label{SectionAsigma}A filtration on the branching semigroup}
Let $h:\Lambda_{\B}\rightarrow \Lambda_{2n}$ be given by $h(\lambda
/ \mu )=f(\mu ,\lambda ^{+})$, where $f(\mu ,\lambda ^{+})$ is the
rearrangement function defined in Section \ref{section4.0}. \ The
image of $h$ is thus all sequences in $\Lambda_{2n}$ ending in zero.
\ As before, we define
the functions $r_{i}:\Lambda_{\B}\rightarrow \mathbb{Z}$ by $%
r_{i}(p)=x_{i}-y_{i}$, where $h(p)=(x_{1},y_{1},...,x_{n},y_{n})$.  Moreover, for $\sigma \in \Sigma$, let $h_{\sigma}$ denote the restriction of $h$ to $\Lambda_{\B}(\sigma)$.

The same argument as in Lemma \ref{semigroupisom} shows:
\begin{lemma}
\label{Lemma5.1}
Let $\sigma \in \Sigma$.  Then $h_{\sigma }:\Lambda_{\B}(\sigma) \rightarrow \Lambda_{2n}$ is
a semigroup embedding, with image the sequences in $\Lambda_{2n}$
ending in zero.  In particular, $h_{\sigma }^{-1}$ is defined on the set of
such sequences.
\end{lemma}
In this section we will only deal with sequences ending in zero, so
$h_{\sigma }^{-1}$ will always be well-defined.  By the above lemma
we endow the $L$-module $$\mathcal{A}_{\sigma
}=\dbigoplus\limits_{p\in \Lambda_{\B}(\sigma)}A^{p}$$with a product
given by Cartan product of irreducible $L$-modules: $A^{p} \otimes
A^{p'} \rightarrow A^{p+p}$.
We now show that $\mathcal{A}_{\sigma}$ is a very naturally occurring $L$-algebra.  Consider the $L$-module $V=U \times W$, where $U=\mathbb{C}^{2} \times \cdots \times \mathbb{C}^{2}$ ($n$ copies) and $W=\mathbb{C} \times \cdots \times \mathbb{C}$ ($n-1$ copies).  Here, $L$ acts on $U$ diagonally, on $W$ trivially, and on the ring of functions $\calli{O}(V)$ by right translation.  Let $t_{1},...,t_{n-1}$ be the standard coordinate functions on $\mathbb{C}%
^{n-1}$. \ Decompose $\mathcal{O}(V)$ into graded components:%
\begin{equation}
\mathcal{O}(V)\cong \dbigoplus\limits_{\substack{ r_{j}\geq 0  \\ j=1,...,n}}%
\dbigoplus\limits_{\substack{ s_{k}\geq 0  \\
k=1,...,n-1}}F^{r_{1}}\otimes \cdots \otimes F^{r_{n}}\otimes
t_{1}^{s_{1}}\cdots t_{n-1}^{s_{n-1}}\text{.} \label{IsomO(V)}
\end{equation}%
This is also a decomposition of $\mathcal{O}(V)$ into irreducible $L$%
-modules.

For $\sigma \in \Sigma$, we can consider $\calli{O}(V)$ as an $\Lambda_{\B}(\sigma)$-graded algebra as follows.
Set $s_{i}:\Lambda_{\B}(\sigma) \rightarrow \mathbb{Z}$ by $s_{i}(p)=y_{i}-x_{i+1}$, where, as usual, $h(p)=(x_{1},y_{1},...,x_{n},y_{n})$.  Then define the $p$-component of $\mathcal{O}(V)$ by:
$$
\mathcal{O}(V)^{p} = F^{r_{1}(p)} \otimes \cdots \otimes F^{r_{n}(p)} \otimes t_{1}^{s_{1}(p)} \cdots t_{n-1}^{s_{n-1}(p)}.
$$
Clearly, we have
$$
\mathcal{O}(V) = \bigoplus_{p \in \Lambda_{\B}(\sigma)} \mathcal{O}(V)^{p}.
$$
One easily proves the following lemma.
\begin{lemma}
\label{OVlemma}
Let $\sigma \in \Sigma$ and regard $\mathcal{O}(V)$ as an $\Lambda_{\B}(\sigma)$-graded $L$-algebra.  Then $\mathcal{A}_{\sigma}$ and $\mathcal{O}(V)$ are isomorphic as $\Lambda_{\B}(\sigma)$-graded $L$-algebras, and the isomorphism is unique up to scalars.
\end{lemma}

The main step in proving Theorem \ref{ExtThm2} is showing that $\mathcal{A%
}_{\sigma }$ and $\B_{\sigma}$ are
isomorphic as $SL_{2}$-algebras, and the isomorphism is unique up to scalars.  We will prove this by induction on a certain filtration of $%
\Lambda_{\B}(\sigma)$, which we now describe.

For $p\in \Lambda_{\B}$ let $p_{\max}=\lambda_{1}$ where $p=\lambda / \mu $ and $\lambda =(\lambda _{1},...,\lambda _{n})$. \ For every $\sigma \in \Sigma $ we define the set
\begin{equation*}
\Lambda_{\B}(\sigma,m)=\{p\in \Lambda_{\B}(\sigma):p_{\max }\leq m\}\text{.}
\end{equation*}
Clearly $\Lambda_{\B}(\sigma,m)$ is finite, $\Lambda_{\B}(\sigma,m-1)\subset\Lambda_{\B}(\sigma,m)$, and $\bigcup_{m\geq0}\Lambda_{\B}(\sigma,m)=\Lambda_{\B}(\sigma).$

\begin{lemma}
\label{peel}Let $m>1$, $\sigma \in \Sigma $, and suppose $p\in
\Lambda_{\B}(\sigma,m)$ satisfies $p_{\max }=m$.
\begin{enumerate}
\item There exist $p^{\prime },p^{\prime \prime }\in \Lambda_{\B}(\sigma,m-1)$ such
that $p=p^{\prime }+p^{\prime \prime }$.
\item Suppose moreover that $\tau \in \Sigma$ and $p \in \Lambda_{\B}(\tau,m)$.  Then there exist $p^{\prime },p^{\prime \prime
}\in \Lambda_{\B}(\sigma,m-1) \cap \Lambda_{\B}(\tau,m-1)$ such that $p=p^{\prime }+p^{\prime \prime }$.
\end{enumerate}
\end{lemma}

\begin{proof}
Let $h_{\sigma }(p)=(z_{1},...,z_{2n})$. \ Define
\begin{equation*}
z_{i}^{\prime }=
\left\{
\begin{array}{rl}
1\text{ if }z_{i}\geq 1  \\
0\text{ if }z_{i}=0
\end{array} \right.
\end{equation*}
and $z_{i}^{\prime \prime }=z_{i}-z_{i}^{\prime }$. \ It's trivial to check
that $\xi ^{\prime }=(z_{1}^{\prime },...,z_{2n}^{\prime }),\xi ^{\prime
\prime }=(z_{1}^{\prime \prime },...,z_{2n}^{\prime \prime })\in \Lambda_{2n}$. \ Let $p^{\prime }=f_{\sigma }^{-1}(\xi ^{\prime })$ and $%
p^{\prime \prime }=f_{\sigma }^{-1}(\xi ^{\prime \prime })$. \ This is
well-defined by Lemma \ref{Lemma5.1}. \ Lemma \ref{Lemma5.1} also
shows that $p=p^{\prime }+p^{\prime \prime }$. \ Since $m>1$%
, $p^{\prime },p^{\prime \prime }\in \Lambda_{\B}(\sigma,m-1)$.  This proves (1).

Let $p^{\prime },p^{\prime \prime }\in \Lambda_{\B}(\sigma,m-1)$
be constructed as in the previous paragraph. \ We must show that $p^{\prime
},p^{\prime \prime }\in \Lambda_{\B}(\sigma)\cap
\Lambda_{\B}(\tau)$. \ If $\sigma =(\sigma _{1}\cdots\sigma _{n-1})$ and $\tau =(\tau
_{1}\cdots\tau _{n-1})$, then every $i$ such
that $\sigma _{i}\not=\tau _{i}$ forces the equality $\mu _{i}=\lambda
_{i+1} $ among the entries of $p$. \ Therefore, if $h_{\sigma
}(p)=(z_{1},...,z_{2n})$ and $\sigma _{i}\not=\tau _{i}$, then $%
z_{2i+1}=z_{2i+2}$. \ Now note that in the definition of $\xi ^{\prime },\xi
^{\prime \prime }$, if $z_{2i+1}=z_{2i+2}$ then $z_{2i+1}^{\prime
}=z_{2i+2}^{\prime }$ and $z_{2i+1}^{\prime \prime }=z_{2i+2}^{\prime \prime
}$. \ Hence the entries of $\xi ^{\prime },\xi ^{\prime \prime }$ satisfy
the same equalities that $h_{\sigma }(p)$ satisfies, which implies that $%
p^{\prime },p^{\prime \prime }\in \Lambda_{\B}(\sigma)\cap
\Lambda_{\B}(\tau)$.  This proves (2).
\end{proof}

\begin{lemma}
\label{megapeel}Let $m>1$, $\sigma \in \Sigma $, and suppose $p\in
\Lambda_{\B}(\sigma,m)$ satisfies $p_{\max }=m$. \ Then there exist $%
q_{1},...,q_{n}\in \Lambda_{\B}(\sigma,m-1)$ such that
\begin{enumerate}
\item $p=q_{1}+\cdots +q_{n}$
\item $A^{q_{i}}$ is an irreducible $SL_{2}$-module.
\item Either $A^{q_{1}}\otimes \cdots \otimes A^{q_{n}}\cong A^{p}$ as $SL_{2}$-modules, or $A^{p}$ is irreducible as an $SL_{2}$%
-module, and the multiplication map $A^{q_{1}}\otimes \cdots \otimes
A^{q_{n}}\rightarrow A^{p}$ is a projection onto the Cartan component $A^{p}$
of $A^{q_{1}}\otimes \cdots \otimes A^{q_{n}}$.
\end{enumerate}
\end{lemma}

\begin{proof}
Let $h_{\sigma }(p)=(x_{1},y_{1},...,x_{n},y_{n})$. \ Define%
\begin{equation*}
\xi _{i}=(\underset{2i-1}{\underbrace{x_{i}-x_{i+1},...,x_{i}-x_{i+1}}}%
,y_{i}-x_{i+1},0,...,0)
\end{equation*}%
for $i=1,...,n-1$, and set $\xi _{n}=(x_{n},...,x_{n},0)$.
The argument breaks into cases.

Case 1: Suppose $h_{\sigma }^{-1}(\xi _{i})\not\in \Lambda
_{\B}(\sigma,m-1)$ for some $i\leq n$. \ Then $x_{i}-x_{i+1}=m$ and therefore
\begin{equation*}
h_{\sigma }(p)=(m,...,m,b,0,...,0)
\end{equation*}%
for some $b\leq m$ in the $(2i)^{th}$ entry. \ Therefore $A^{p}$ is
irreducible as an $SL_{2}$-module.

Now choose $\xi ^{\prime },\xi ^{\prime \prime }$ as in the proof of Lemma %
\ref{peel} and consider the associated $p^{\prime },p^{\prime \prime
}$. \ By the lemma $p^{\prime },p^{\prime \prime }\in \Lambda
_{\B}(\sigma,m-1)$. \ Moreover, by our construction of $\xi ^{\prime
},\xi ^{\prime \prime }$ from $\xi $, $\calli{A}^{p^{\prime
}},\calli{A}^{p^{\prime \prime }}$ are irreducible $SL_{2}$-modules.
\ Therefore the map $A^{p^{\prime }}\otimes A^{p^{\prime \prime
}}\rightarrow A^{p}$ is a projection onto the Cartan component of
$A^{p^{\prime }}\otimes A^{p^{\prime \prime }}$, and the
lemma is satisfied with $q_{1}=p^{\prime },q_{2}=p^{\prime \prime }$, and $%
q_{i}=0$ for $i>2$.


Case 2: Suppose that $h_{\sigma }^{-1}(\xi _{i})\in \Lambda
_{\B}(\sigma ,m-1)$ for $i=1,...,n$. \ Then set $q_{i}=h_{\sigma }^{-1}(\xi
_{i}) $. \ Since $\xi =\xi _{1}+\cdots +\xi _{n}$, by Lemma \ref%
{semigroupisom},{\small \ }$p=q_{1}+\cdots +q_{n}$. \ By the definition of $%
\xi _{i}$ we also have that%
\begin{equation*}
A^{q_{i}}=F^{0}\otimes \cdots \otimes \underset{i^{th}}{\underbrace{%
F^{x_{i}-y_{i}}}}\otimes \cdots \otimes F^{0}\text{.}
\end{equation*}%
Therefore $A^{q_{i}}$ is an irreducible $SL_{2}$-module, and $%
A^{q_{1}}\otimes \cdots \otimes A^{q_{n}}\cong A^{p}$.
\end{proof}

\begin{remark}
In the proof of Lemma \ref{megapeel} all we used was the $SL_{2}$%
-module structure of $A^{p}$. \ Therefore, by Corollary \ref%
{CorollaryWY}, the statement holds with $A^{p}$ replaced by $%
W^{p}$ and $A^{q_{i}}$ replaced by $W^{q_{i}}$.
\end{remark}

\begin{lemma}
\label{kerlemma}Let $m>1$, $\sigma \in \Sigma $, and suppose $p\in
\Lambda_{\B}(\sigma,m)$ satisfies $p_{\max }=m$. \ Let $%
q_{1},...,q_{n}\in \Lambda_{\B}(\sigma,m-1)$ be given as in
Lemma \ref{megapeel}. \ Suppose also we are given $SL_{2}$%
-isomorphisms $\phi _{i}:W^{q_{i}}\rightarrow A^{q_{i}}$ for $%
i=1,...,n$. \ Let%
\begin{eqnarray*}
K &=&\ker (W^{q_{1}}\otimes \cdots \otimes W^{q_{n}}%
\overset{\tau }{\rightarrow }W^{p}) \\
J &=&\ker (A^{q_{1}}\otimes \cdots \otimes A^{q_{n}}\overset{\kappa }{%
\rightarrow }A^{p})
\end{eqnarray*}%
be the kernels of the multiplication maps coming from the rings $\calli{B}%
_{\sigma }$ and $\calli{A}_{\sigma }$, which we denote here by $\tau$ and $\kappa$. \ Set $\phi =\phi _{1}\otimes \cdots \otimes
\phi _{n}$. \ Then $\phi (K)=J$. \ Consequently, there is an $SL_{2}$-isomorphism $\psi :W^{p}\rightarrow A^{p}$ making the following
diagram commute:
\begin{equation*}
\begin{array}{ccc}
W^{q_{1}}\otimes \cdots \otimes W^{q_{n}} &
\longrightarrow & W^{p} \\
\downarrow {\small \phi } &  & \downarrow {\small \psi } \\
A^{q_{1}}\otimes \cdots \otimes A^{q_{n}} & \longrightarrow & A^{p}%
\end{array}
\label{diag4}
\end{equation*}%
\end{lemma}

\begin{proof}
Clearly $\kappa $ is surjective. \ By Proposition \ref{SpnSurjectivityCorollary},{\small \ }$\tau
$ is surjective. \ Therefore we have the following diagram:
\begin{equation*}
\begin{array}{ccccccccc}
0 & \longrightarrow & K & \longrightarrow & W^{q_{1}}\otimes
\cdots \otimes W^{q_{n}} & \longrightarrow & W^{p} &
\longrightarrow & 0 \\
&  &  &  & \downarrow {\small \phi } &  &  &  &  \\
0 & \longrightarrow & J & \longrightarrow & A^{q_{1}}\otimes \cdots \otimes
A^{q_{n}} & \longrightarrow & A^{p} & \longrightarrow & 0%
\end{array}%
\end{equation*}%
According to Lemma \ref{megapeel} there are two possibilities. \ Either $%
A^{q_{1}}\otimes \cdots \otimes A^{q_{n}}\cong A^{p}$ as $SL_{2}$%
-modules, or $A^{p}$ is irreducible as a $SL_{2}$-module, and the
multiplication map $A^{q_{1}}\otimes \cdots \otimes A^{q_{n}}\rightarrow
A^{p}$ is a projection onto the Cartan component $A^{p}$ of $%
A^{q_{1}}\otimes \cdots \otimes A^{q_{n}}$. \ If $A^{q_{1}}\otimes \cdots
\otimes A^{q_{n}}\cong A^{p}$ then $J=\{0\}$, and by the above remark $%
W^{q_{1}}\otimes \cdots \otimes W^{q_{n}}\cong W^{p}$. \ Therefore $K=\{0\}$, and so clearly $\phi (K)=J$.

In the other case, $A^{p}\cong W^{p}$ are irreducible
$SL_{2}$-modules. \ Choose $k$ so that $A^{p}\cong W^{p}\cong
F^{k}$. \ Since the maps $\kappa ,\tau $ are both projections onto
the Cartan component $F^{k}$, their kernels are given as sums of
$SL_{2}$-isotypic components
\begin{eqnarray*}
K &=&\dsum\limits_{n<k}(W^{q_{1}}\otimes \cdots \otimes W^{q_{n}})[j] \\
J &=&\dsum\limits_{j<k}(A^{q_{1}}\otimes \cdots \otimes A^{q_{n}})[j]\text{.}
\end{eqnarray*}%
Since $\phi $ intertwines the $SL_{2}$-action, $\phi (K)\subset J$%
. \ Moreover, $\kappa $ and $\tau $ are both Cartan multiplications of the
same $SL_{2}$-modules, and so $\dim K=\dim J$. \ Therefore $\phi(K)=J$.
\end{proof}

\begin{lemma}
\label{cubic}Let $m>1$, $\sigma \in \Sigma $, and suppose we are given $p^{\prime },p^{\prime \prime
},q^{\prime },q^{\prime \prime }\in \Lambda_{\B}(\sigma,m)$
such that $p^{\prime }+p^{\prime \prime }=q^{\prime }+q^{\prime \prime }$.
\ Then there exist $t^{\prime },t^{\prime \prime },r^{\prime },r^{\prime
\prime }\in \Lambda_{\B}(\sigma,m)$ such that
\begin{eqnarray*}
t^{\prime }+r^{\prime } =p^{\prime },  t^{\prime \prime }+r^{\prime \prime } =p^{\prime \prime } \\
t^{\prime }+t^{\prime \prime } =q^{\prime },  r^{\prime }+r^{\prime \prime } =q^{\prime \prime }
\end{eqnarray*}%
\end{lemma}
\begin{proof}
For some nonnegative
integers $n_{i}^{\prime },n_{i}^{\prime \prime },m_{i}^{\prime
},m_{i}^{\prime \prime }$ we have%
\begin{eqnarray*}
h_{\sigma }(p^{\prime }) &=&\tsum\nolimits_{i=1}^{2n}n_{i}^{\prime }\varpi
_{i}\text{, }h_{\sigma }(p^{\prime \prime
})=\tsum\nolimits_{i=1}^{2n}n_{i}^{\prime \prime }\varpi _{i} \\
h_{\sigma }(q^{\prime }) &=&\tsum\nolimits_{i=1}^{2n}m_{i}^{\prime }\varpi
_{i}\text{, }h_{\sigma }(q^{\prime \prime
})=\tsum\nolimits_{i=1}^{2n}m_{i}^{\prime \prime }\varpi _{i}
\end{eqnarray*}

Define:%
\begin{eqnarray*}
\tau ^{\prime } &=&\sum_{i=1}^{2n}(m_{i}^{\prime \prime }-\min
(n_{i}^{\prime \prime },m_{i}^{\prime \prime }))\varpi _{i} \\
\rho ^{\prime } &=&\sum_{i=1}^{2n}(n_{i}^{\prime }-m_{i}^{\prime \prime
}+\min (n_{i}^{\prime \prime },m_{i}^{\prime \prime }))\varpi _{i} \\
\tau ^{\prime \prime } &=&\sum_{i=1}^{2n}\min (n_{i}^{\prime \prime
},m_{i}^{\prime \prime })\varpi _{i} \\
\rho ^{\prime \prime } &=&\sum_{i=1}^{2n}(n_{i}^{\prime \prime }-\min
(n_{i}^{\prime \prime },m_{i}^{\prime \prime }))\varpi _{i}\text{.}
\end{eqnarray*}%
Clearly $\tau ^{\prime },\tau ^{\prime \prime },\rho ^{\prime \prime }\in \Lambda_{2n}$. Since $p^{\prime }+p^{\prime
\prime }=q^{\prime }+q^{\prime \prime }$, $\rho ^{\prime } \in \Lambda_{2n}$ as well. \
Now, note that
\begin{eqnarray*}
h_{\sigma }(p^{\prime }) =\tau ^{\prime }+\rho ^{\prime },
h_{\sigma }(p^{\prime \prime }) =\tau ^{\prime \prime }+\rho ^{\prime
\prime } \\
h_{\sigma }(q^{\prime }) =\rho ^{\prime }+\rho ^{\prime \prime },
h_{\sigma }(q^{\prime \prime }) =\tau ^{\prime }+\tau ^{\prime \prime }%
\text{.}
\end{eqnarray*}
Set
\begin{eqnarray*}
t^{\prime } &=&h_{\sigma }^{-1}(\tau ^{\prime })\text{, \ }t^{\prime \prime
}=h_{\sigma }^{-1}(\tau ^{\prime \prime }) \\
r^{\prime } &=&h_{\sigma }^{-1}(\rho ^{\prime })\text{, \ }r^{\prime
\prime }=h_{\sigma }^{-1}(\rho ^{\prime \prime })
\end{eqnarray*}%
Since $h_{\sigma }$ is a semigroup isomorphism and $p^{\prime },p^{\prime
\prime },q^{\prime },q^{\prime \prime }\in \Lambda_{\B}(\sigma)$, it follows that $t^{\prime },t^{\prime \prime },r^{\prime
},r^{\prime \prime }\in \Lambda_{\B}(\sigma)$ and they
satisfy the desired equations.
\end{proof}

\subsection{\label{section5.3}Proof of Proposition \ref{CorIsomMtoA}}

\begin{definition}
\label{compatible}Let $\mathcal{F}=\{\phi _{p}:W^{p}\rightarrow
A^{p}\}_{p\in \Lambda_{\B}}$ be a family of $SL_{2}$%
-isomorphisms indexed by $\Lambda_{\B}$. \ Then $\mathcal{F}$ is a
\textbf{compatible family} if for any $%
\sigma \in \Sigma $ and $p^{\prime },p^{\prime \prime }\in
\Lambda_{\B}(\sigma)$ the following diagram commutes:
\begin{equation}
\begin{array}{ccc}
W^{p^{\prime }}\otimes W^{p^{\prime \prime }} &
\longrightarrow & W^{p^{\prime }+p^{\prime \prime }} \\
\downarrow &  & \downarrow \\
A^{p^{\prime }}\otimes A^{p^{\prime \prime }} & \longrightarrow &
A^{p^{\prime }+p^{\prime \prime }}%
\end{array}
\label{DI}
\end{equation}
Here the vertical maps are given by $\phi _{p^{\prime }}\otimes \phi _{p^{\prime
\prime }}$ and $\phi _{p^{\prime }+p^{\prime \prime }}$, and the
horizontal maps are the product maps in the rings $\B%
_{\sigma }$ and $\mathcal{A}_{\sigma }$. \
\end{definition}


\begin{proposition}
\label{BigProp}There exists a compatible family $\mathcal{F}=\{\phi _{p}:%
W^{p}\rightarrow A^{p}\}_{p\in \Lambda_{\B}}$ of $SL_{2}$-isomorphisms. \ Moreover, each map $\phi _{p}\in \mathcal{F}$
is unique up to scalar.
\end{proposition}

\begin{proof}
For a nonnegative integer $m$ set $\Lambda_{\B}(m)=\{p\in
\Lambda_{\B}:p_{\max }\leq m\}$. \ We first prove by induction on
$m$ that there is a family of $SL_{2}$ isomorphisms
\begin{equation*}
\mathcal{F}_{m}=\{\phi _{p}:W^{p}\rightarrow A^{p}\}_{p\in
\Lambda_{\B}(m)}
\end{equation*}%
such that for any $p\in \Lambda_{\B}(m)$, $\sigma \in \Sigma $,
and $p^{\prime },p^{\prime \prime }\in \Lambda_{\B}(\sigma )$
such that $p=p^{\prime }+p^{\prime \prime }$, diagram (\ref{DI}) commutes.

For the base case we construct $\mathcal{F}_{1}$. \ If $p_{\max }=0$ then $%
p=p_{0}=0/0$. \ We define $\phi _{p_{0}}:W^{p_{0}}\rightarrow
A^{p_{0}}$ by $1\in W^{p_{0}}\mapsto 1\otimes \cdots \otimes
1\in A^{p_{0}}$. \ Of course, if $p^{\prime }+p^{\prime \prime }=p_{0}$ then
$p^{\prime }=p^{\prime \prime }=p_{0}$ and (\ref{DI}) trivially commutes. \
Suppose now that $p_{\max }=1$. \ Then $p=\lambda / \mu $ with $\lambda $ a
fundamental weight, and $\mu $ either zero or a fundamental weight. \ In any
case, $A^{p}$ is an irreducible $SL_{2}$-module, and by Corollary %
\ref{CorollaryWY}, $W^{p}\cong A^{p}$ as $SL_{2}$-modules. \ We
choose an $SL_{2}$-isomorphism $\phi _{p}:W^{p}\rightarrow A^{p}$. \
By Schur's Lemma, $\phi _{p}$ is unique up to scalar. \ Now suppose
$\sigma \in \Sigma $, $p^{\prime },p^{\prime \prime }\in
\Lambda_{\B}(\sigma)$, and $p^{\prime }+p^{\prime \prime }=p$. \
Then either $p_{\max }^{\prime }=0$ or $p_{\max }^{\prime \prime
}=0$. \ Assume, without loss of generality, that $p_{\max }^{\prime
}=0$. \ Then $p^{\prime }=p_{0}$, and by our construction of $\phi
_{p_{0}}$, diagram (\ref{DI}) commutes. \ Set $\mathcal{F}_{1}=\{\phi _{p}:%
W^{p}\rightarrow A^{p}\}_{p\in \Lambda_{\B}(1)}$; this completes the
base case. \

Let $m>1$ and suppose that $\mathcal{F}_{m-1}$ exists and satisfies the
desired properties. \ We must construct $\mathcal{F}_{m}$. \ For $p\in \Lambda_{\B}(m)$
such that $p_{\max }<m$, there exists $\phi _{p}\in \mathcal{F}_{m-1}$ by
hypothesis. \ We include these $\phi _{p}$ in $\mathcal{F}_{m}$.  For such $p$ we have the following: if $%
\sigma \in \Sigma $, $p^{\prime },p^{\prime \prime }\in \Lambda_{\B}(\sigma )$, and $p=p^{\prime }+p^{\prime \prime }$, then diagram (\ref{DI}%
) commutes. \ Indeed, $p=p^{\prime }+p^{\prime \prime }$ implies that $%
p^{\prime },p^{\prime \prime }\in \Lambda_{\B}(m-1)$. \ Therefore
$\phi _{p^{\prime }}$ and $\phi _{p^{\prime \prime }}$ are also obtained
from $\mathcal{F}_{m-1}$, and diagram (\ref{DI}) commutes by hypothesis. \

Suppose $p\in \Lambda_{\B}(m)$ satisfies $p_{\max }=m$. \ Choose
an order type $\sigma \in \Sigma $ such that $p\in \Lambda%
_{B}(\sigma)$. \ Note that $\sigma $ may not be unique. \ Choose $%
q_{1},...,q_{n}\in \Lambda_{\B}(\sigma,m-1)$ by Lemma \ref%
{megapeel}. \ Now apply Lemma \ref{kerlemma} to obtain an $SL_{2}$%
-isomorphism $\psi :W^{p}\rightarrow A^{p}$ such that the
following diagram commutes:
\begin{equation}
\begin{array}{ccc}
W^{q_{1}}\otimes \cdots \otimes W^{q_{n}} &
\longrightarrow & W^{p} \\
\downarrow {\small \phi } &  & \downarrow {\small \psi } \\
A^{q_{1}}\otimes \cdots \otimes A^{q_{n}} & \longrightarrow & A^{p}%
\end{array}
\label{DII}
\end{equation}%
where $\phi =\phi _{q_{1}}\otimes \cdots \otimes \phi _{q_{n}}$. \ We now
show that (i) if $p^{\prime },p^{\prime \prime }\in \Lambda%
_{\B}(\sigma)$ satisfy $p=p^{\prime }+p^{\prime \prime }$ then
\begin{equation}
\begin{array}{ccc}
W^{p^{\prime }}\otimes W^{p''} &
\longrightarrow & W^{p} \\
\downarrow &  & \downarrow \psi \\
A^{p^{\prime }}\otimes A^{p^{\prime \prime }} & \longrightarrow & A^{p}%
\end{array}
\label{DII.1}
\end{equation}%
commutes, (ii) $\psi $ is independent of the choice of $q_{1},...,q_{n}$,
and (iii) $\psi $ is independent of the choice of $\sigma $.

First note that (i) implies (ii). \ Indeed, suppose $q_{1}^{\prime
},...,q_{n}^{\prime }\in \Lambda_{\B}(\sigma,m-1)$ is another
collection of shapes satisfying the conditions of Lemma
\ref{megapeel}, and $\psi ^{\prime }:W^{p}\rightarrow A^{p}$ is the
associated $SL_{2}$-isomorphism obtained by Lemma \ref{kerlemma}. \
By (i) both $\psi $ and $\psi ^{\prime }$ would make (\ref{DII.1})
commute. \ But since all the maps in the diagram are surjective,
there is a unique map making (\ref{DII.1}) commute. \ Therefore
$\psi =\psi ^{\prime }$.

Now we prove (i).
If $%
p_{\max }^{\prime }=m$ (resp. $p_{\max }^{\prime \prime }=m$) then $%
p^{\prime \prime }=p_{0}$ (resp. $p^{\prime }=p_{0}$), and (\ref{DII.1})
commutes by our choice of $\phi _{p_{0}}$. \ Therefore we may assume that $%
p_{\max }^{\prime },p_{\max }^{\prime \prime }<m$. \ By renumbering the $%
q_{j}$ if necessary, we may assume that $(q_{1})_{\max }\not=0$. \ Let $%
q^{\prime }=q_{1}$ and $q^{\prime \prime }=q_{2}+\cdots +q_{n}$. \ Then $%
q^{\prime },q^{\prime \prime }\in \Lambda_{\B}(\sigma,m-1)$ and $%
q^{\prime }+q^{\prime \prime }=p$. \ By inductive hypothesis the following
diagram commutes:
\begin{equation}
\begin{array}{ccc}
W^{q_{1}}\otimes \cdots \otimes W^{q_{n}} &
\longrightarrow & W^{q^{\prime }}\otimes W^{q^{\prime
\prime }} \\
\downarrow &  & \downarrow \\
A^{q_{1}}\otimes \cdots \otimes A^{q_{n}} & \longrightarrow & A^{q^{\prime
}}\otimes A^{q^{\prime \prime }}%
\end{array}
\label{DIII}
\end{equation}%
where the vertical map on the left is $\phi =\phi _{q_{1}}\otimes \cdots
\otimes \phi _{q_{n}}$, and the one on the right is $\phi _{q^{\prime
}}\otimes \phi _{q^{\prime \prime }}$. \ Combining (\ref{DII}) and (\ref%
{DIII}) and the fact that all the maps are surjective (Proposition \ref%
{SpnSurjectivityCorollary}), we conclude that
\begin{equation}
\begin{array}{ccc}
W^{q^{\prime }}\otimes W^{q^{\prime \prime }} &
\longrightarrow & W^{p} \\
\downarrow &  & \downarrow {\small \psi } \\
A^{q^{\prime }}\otimes A^{q^{\prime \prime }} & \longrightarrow & A^{p}%
\end{array}
\label{DIV}
\end{equation}%
commutes.

Since $p'+p''=q'+q''$, by Lemma \ref{cubic} there exist $t^{\prime
},t^{\prime \prime },r^{\prime },r^{\prime \prime }\in
\Lambda_{\B}(\sigma,m)$ such that
\begin{eqnarray*}
t^{\prime }+r^{\prime } = p^{\prime },
t^{\prime \prime }+r^{\prime \prime } =p^{\prime \prime } \\
t^{\prime }+t^{\prime \prime } =q^{\prime },
r^{\prime }+r^{\prime \prime } =q^{\prime \prime }\text{.}
\end{eqnarray*}%
Consider the following diagram:
\begin{equation*}
\xymatrix{
& W^{r^{\prime }}\otimes W^{t^{\prime }}\otimes W
^{r^{\prime \prime }}\otimes W^{t^{\prime \prime }} \ar[dl]\ar[rr]\ar'[d][dd]
& & W^{q^{\prime }}\otimes W^{q^{\prime \prime }} \ar[dl]\ar[dd]
\\
W^{p^{\prime }}\otimes W^{p''} \ar[rr]\ar[dd]
& & W^{p} \ar[dd]
\\
& A^{r^{\prime }}\otimes A^{t^{\prime }}\otimes A^{r^{\prime \prime }}\otimes
A^{t^{\prime \prime }} \ar[dl]\ar'[r][rr]
& & A^{q^{\prime
}}\otimes A^{q^{\prime \prime }} \ar[dl]
\\
A^{p^{\prime }}\otimes A^{p^{\prime \prime }} \ar[rr]
& & A^{p}
}
\end{equation*}
The top square commutes by associativity of the product in $\calli{B}%
_{\sigma }$. \ The left and back squares commute by inductive hypothesis. \ The right square
commutes since it is the diagram (\ref{DIV}). \ The bottom square commutes by
associativity of the product in $\calli{A}_{\sigma }$. \ By chasing this diagram and repeatedly
using Proposition \ref{SpnSurjectivityCorollary}, it follows that the front square
commutes. \ This proves (ii).

We now prove (iii), namely that $\psi $ is independent of $\sigma $. \
Indeed, suppose $\tau \in \Sigma $ is another order type such that $p\in
\Lambda_{\B}(\tau)$. \ By the above argument we obtain an $SL_{2}$ isomorphism $\zeta :W^{p}\rightarrow A^{p}$ such that
(\ref{DII.1}) commutes for all $p^{\prime },p^{\prime \prime }\in
\Lambda_{\B}(\tau)$. By Lemma \ref{peel}(2) there exist $p^{\prime },p^{\prime \prime }\in \Lambda_{\B}(\sigma)\cap \Lambda_{\B}(\tau)$ such that $p=p^{\prime
}+p^{\prime \prime }$. \ Therefore both $\psi $ and $\zeta $ make the
following diagram commute:
\begin{equation*}
\begin{array}{ccc}
W^{p^{\prime }}\otimes W^{p''} &
\longrightarrow & W^{p} \\
\downarrow &  & \downarrow {\small \psi ,\zeta } \\
A^{p^{\prime }}\otimes A^{p^{\prime \prime }} & \longrightarrow & A^{p}%
\end{array}%
\end{equation*}%
Hence $\psi =\zeta $.

At this point we've shown for any $p\in \Lambda_{\B}(m)$ there is
a canonical $SL_{2}$ isomorphism $\psi :W^{p}\rightarrow A^{p}$
satisfying the property: for any $\sigma \in \Sigma $ and $p^{\prime
},p^{\prime \prime }\in \Lambda_{\B}(\sigma)$ such that $%
p=p^{\prime }+p^{\prime \prime }$, diagram (\ref{DII.1}) commutes. \ Set $%
\phi _{p}=\psi $ and define $\mathcal{F}_{m}=\{\phi _{p}:W^{p}\rightarrow A^{p}\}_{p\in \Lambda_{\B}(m)}$. \ This completes
the induction.

Let $\mathcal{F}=\bigcup_{m=1}^{\infty }\mathcal{F}_{m}$.  By construction, $\calli{F}$ is a compatible family of $SL_{2}$-isomorphisms.
This completes the proof of the first statement of the
proposition.

Now suppose $\widetilde{\mathcal{F}}=\{\widetilde{\phi }_{p}:%
W^{p}\rightarrow A^{p}\}_{p\in \Lambda_{\B}}$ is another compatible
family of $SL_{2}$-isomorphisms. \ We will show by induction on
$p_{\max }$ that there exist a set of nonzero scalars $\{c_{p}\in
\mathbb{C}^{\times }:p\in \Lambda_{\B}\}$, such that $\phi
_{p}=c_{p}\widetilde{\phi }_{p}$ for all $p\in \Lambda_{\B}$.

We already noted that by Schur's Lemma each isomorphism $\phi _{p}$ with $%
p_{\max }=1$ is unique up to scalar. \ Therefore there exist $c_{p}\in
\mathbb{C}^{\times }$ such that%
\begin{equation}
\phi _{p}=c_{p}\widetilde{\phi }_{p}  \label{uptoscalar}
\end{equation}%
for all $p$ with $p_{\max }=1$. \ Let $m>1$. \ Suppose now that there exist
scalars so that (\ref{uptoscalar}) holds for all $p\in \Lambda_{\B}$ such that $p_{\max }<m$. \ Let $p\in \Lambda_{\B}$ with $p_{\max
}=m$. \ Choose some $\sigma \in \Sigma $ such that $p\in \Lambda_{\B}(\sigma)$. \ By Lemma \ref{peel}, there exist $p^{\prime },p^{\prime
\prime }\in \Lambda_{\B}(\sigma,m-1)$ such that $p=p^{\prime
}+p^{\prime \prime }$. \ Then by the compatibility of $\mathcal{F}$ the
following diagram commutes:
\begin{equation}
\begin{array}{ccc}
W^{p^{\prime }}\otimes W^{p''} &
\longrightarrow & W^{p} \\
\downarrow &  & \downarrow \\
A^{p^{\prime }}\otimes A^{p^{\prime \prime }} & \longrightarrow & A^{p}%
\end{array}
\label{DV}
\end{equation}%
where the vertical maps are $\phi _{p^{\prime }}\otimes \phi _{p^{\prime
\prime }}$ and $\phi _{p}$. \ By hypothesis, $$\phi _{p^{\prime }}\otimes
\phi _{p^{\prime \prime }}=c_{p^{\prime }}c_{p^{\prime \prime }}\widetilde{%
\phi }_{p^{\prime }}\otimes \widetilde{\phi }_{p^{\prime \prime
}}.$$ \
Therefore (\ref{DV}) commutes with the vertical maps replaced by $\widetilde{%
\phi }_{p^{\prime }}{\tiny \otimes }\widetilde{\phi }_{p^{\prime \prime }}$
and $\frac{1}{c_{p^{\prime }}c_{p^{\prime \prime }}}\phi _{p}$. \ Hence $%
\frac{1}{c_{p^{\prime }}c_{p^{\prime \prime }}}\phi _{p}=\widetilde{\phi }%
_{p}$, or, in other words, $c_{p}=c_{p^{\prime }}c_{p^{\prime \prime }}$ and
$\phi _{p}=c_{p}\widetilde{\phi }_{p}$. \ This completes the induction, and
shows that $\phi _{p}\in \mathcal{F}$ is unique up to scalar.
\end{proof}

We can now prove Proposition \ref{CorIsomMtoA}.  Indeed, let
$\mathcal{F}=\{\phi _{p}:W^{p}\rightarrow A^{p}\}_{p\in
\Lambda_{\B}}$ be a compatible
family of $SL_{2}$-isomorphisms guaranteed by the above proposition.  Define a map%
\begin{equation}
\label{phisigma}
\phi _{\sigma }:\calli{B}_{\sigma}\rightarrow \mathcal{A}%
_{\sigma }
\end{equation}%
by%
\begin{equation*}
\phi _{\sigma }|_{W^{p}}=\phi _{p}
\end{equation*}%
for all $p\in \Lambda_{\B}(\sigma)$, and extend linearly. \ Since
$\mathcal{F}$ is a compatible family, $\phi _{\sigma }$ is an
isomorphism of $SL_{2}$-algebras. \ Indeed, the commutativity of
diagram (\ref{DI}) means precisely that $\phi _{\sigma }$ is an
algebra homomorphism.

Now suppose $\widetilde{\phi}_{\sigma}:\B_{\sigma} \rightarrow
\mathcal{A}_{\sigma}$ is some other isomorphism of
$SL_{2}$-algebras.  Set $\widetilde{\phi} _{p}=\widetilde{\phi}
_{\sigma }|_{W^{p}}$. Then
$$
\{ \widetilde{\phi}_{p} : p \in \Lambda_{\B}(\sigma)\}
$$
is a compatible family of $SL_{2}$ isomorphisms.  Therefore, by the
above proposition, there scalars $c_{p}$ such that $\phi
_{p}=c_{p}\widetilde{\phi}_{p}$.  Therefore the graded components of
$\widetilde{\phi}_{\sigma}$ are scalar multiples of the graded
components of $\phi_{\sigma}$, i.e. $\phi_{\sigma}$ is unique up to
scalars.  This proves part (1) of Proposition \ref{CorIsomMtoA}.

To prove part (2) we define a representation of $L$ on
$\calli{B}_{\sigma}$, denoted
$\Phi _{\sigma }$, by the formula%
\begin{equation*}
\Phi _{\sigma }(g)=\phi _{\sigma }^{-1}\circ \theta _{\sigma
}(g)\circ \phi _{\sigma }.
\end{equation*}%
Here $g\in L$, $\phi _{\sigma }$ is the algebra isomorphism coming
from Proposition \ref{BigProp} as in (\ref{phisigma}), and
$\theta_{\sigma}$ is the action of $L$ on $\calli{A}_{\sigma}$
defined in (\ref{defAsigma}). We must show that $\Phi _{\sigma }$ is
independent of the choice of $\phi _{\sigma }$.

Indeed, suppose we are given another isomorphism
$\widetilde{\phi}_{\sigma}$ and use it to define the corresponding
representation of $L$ on $\B_{\sigma}$, which we denote
$\widetilde{\Phi}_{\sigma}$.  Now, choose scalars $c_{p}$ as above.
Then for any $g\in L$ and $p \in \Lambda_{\B}(\sigma)$,
\begin{eqnarray*}
\Phi _{\sigma}(g)|_{W^{p}} &=&\phi _{p}^{-1}\circ \theta
_{p}(g)\circ \phi _{p} \\
&=&(c_{p}\widetilde{\phi}_{p})^{-1}\circ \theta _{p}(g)\circ (c_{p}%
\widetilde{\phi}_{p}) \\
&=&(\widetilde{\phi}_{p})^{-1}\circ \theta _{p}(g)\circ
\widetilde{\phi}_{p}
\\
&=&\widetilde{\Phi} _{\sigma}(g)|_{W^{p}}.
\end{eqnarray*}%
Therefore $\Phi _{\sigma}=\widetilde{\Phi} _{\sigma}$.  This proves
part (2) of Proposition \ref{CorIsomMtoA}.

\subsection{\label{SectionProofBigThe}Proof of Theorem \ref{ExtThm2}}
%
%

\textbf{Existence}:
Consider the representations $(\Phi _{\sigma },\B_{\sigma})$ of $L$
from Proposition \ref{CorIsomMtoA}, and let $mathcal{F}$ be the
compatible family guaranteed by Proposition \ref{BigProp}. These
representations satisfy four desirable properties, all of which are
almost tautologies.

(i) \ For any $p\in \Lambda_{\B}(\sigma)$, $W^{p}$ is
an irreducible $L$-submodule isomorphic to $\bigotimes_{i=1}^{n}F^{r_{i}(p)}$%
. \ Indeed, by definition of $\Phi _{\sigma }$, $\phi
_{p}:W^{p}\rightarrow A^{p}$ is an isomorphism of $L$-modules.

(ii) \ $L$ acts as algebra automorphisms on $\B_{\sigma}$.\ In other
words, we claim that for $p,p^{\prime }\in \Lambda_{\B}(\sigma)$,
the product map, $W^{p}\otimes W^{p^{\prime }}\rightarrow
W^{p+p^{\prime }}$, is a homomorphism of $L$-modules. \ Indeed, by
the compatibility of $\mathcal{F}$, the product map factors as
follows:
\begin{equation*}
\begin{array}{ccc}
W^{p}\otimes W^{p^{\prime }} & \longrightarrow &
W^{p+p^{\prime }} \\
\downarrow &  & \uparrow {\tiny \phi }_{p+p^{\prime }}^{-1} \\
A^{p}\otimes A^{p^{\prime }} & \longrightarrow & \calli{A}_{p+p^{\prime }}%
\end{array}%
\end{equation*}%
Since the three lower maps are $L$-module morphisms, it follows that the top
map is too.

(iii) \ $Res_{SL_{2}}^{L}(\Phi_{\sigma})$ is the natural action of
$SL_{2}$ on $\calli{B}_{\sigma}$.  In other words, for $x\in SL_{2}$
\begin{equation*}
x|_{\calli{B}_{\sigma}}=\Phi _{\sigma }(\delta (x))
\end{equation*}%
where $x|_{\calli{B}_{\sigma}}$ denotes the natural action of $%
x $ on $\calli{B}_{\sigma}$ and $\delta$ is the diagonal embedding
of $SL_{2}$ into $L$. \ Indeed, $\phi _{\sigma }$
intertwines the natural action of $SL_{2}$ on $\B_{\sigma }$ with the diagonal $SL_{2}$-action on $\mathcal{A}%
_{\sigma }$. \ This means that%
\begin{equation*}
\phi _{\sigma }\circ (x|_{\calli{B}_{\sigma}})=\theta _{\sigma
}(\delta (x))\circ \phi _{\sigma }\text{.}
\end{equation*}%
Therefore
\begin{eqnarray*}
\Phi _{\sigma }(\delta (x)) &=&\phi _{\sigma }^{-1}\circ \theta _{\sigma
}(\delta (x))\circ \phi _{\sigma } \\
&=&\phi _{\sigma }^{-1}\circ \phi _{\sigma }\circ (x|_{\B%
_{\sigma }}) \\
&=&x|_{\calli{B}_{\sigma}}\text{.}
\end{eqnarray*}

(iv) \ For any $\sigma _{1},\sigma _{2}\in \Sigma $ and $g\in L$%
\begin{equation}
\Phi _{\sigma _{1}}(g)|_{\B_{\sigma _{1}}\cap
\B_{\sigma _{2}}}=\Phi _{\sigma _{2}}(g)|_{\B
_{\sigma _{1}}\cap \B_{\sigma _{2}}}\text{.}
\label{welldef}
\end{equation}%
Indeed, suppose $p\in \Lambda_{\B}(\sigma _{1})\cap
\Lambda_{\B}(\sigma _{2})$. \ Then $\phi _{\sigma
_{1}}|_{W^{p}}=\phi _{p}=\phi _{\sigma _{2}}|_{W^{p}}$ from which
(\ref{welldef}) immediately follows.

We now construct from $\mathcal{F}$ the representation $(\Phi _{\mathcal{F}},%
\B)$ of $L$ satisfying the conditions of Theorem \ref{ExtThm2}. \ Let $g\in L$. \ Define $\Phi _{\mathcal{F}%
}(g)$ on $\B$ by%
\begin{equation*}
\Phi _{\mathcal{F}}(g)|_{\calli{B}_{\sigma}}=\Phi _{\sigma }(g)%
\text{.}
\end{equation*}%
By (\ref{welldef}) this is well-defined, and since $\sum_{\sigma }
\B_{\sigma }=\B$, this gives an action of $L$
on all of $\B$. \ Moreover, by properties (i) and (ii), $%
(\Phi _{\mathcal{F}},\B)$ satisfies the conditions of
Theorem \ref{ExtThm2}. \ By property (iii) $(\Phi _{\mathcal{F}},
\B)$ extends the natural action of $SL_{2}$ on $
\B$.

\textbf{Uniqueness}: \ Suppose $(\Phi, \B)$ is some representation
of $L$ satisfying the conditions of Theorem \ref{ExtThm2}.
\ By Proposition \ref{CorIsomMtoA}, to show uniqueness it suffices to show that there exists a compatible family $\widetilde{\mathcal{F}%
}=\{\widetilde{\phi }_{p}:W^{p}\rightarrow A^{p}\}_{p\in
\Lambda_{\B}}$ such that $\Phi =\Phi _{\widetilde{\mathcal{F}}}$.

By condition (1) of the theorem, $W^{p}$ is isomorphic to $A^{p}$ as
$L$-modules for every $p\in \Lambda_{\B}$.
In particular, we can choose a set of $L$-isomorphisms $\{\widetilde{%
\phi }_{p}:W^{p}\rightarrow A^{p}\}_{p_{\max }=1}$. \ From this
set we construct a compatible family $\widetilde{\mathcal{F}}=\{\widetilde{%
\phi }_{p}:W^{p}\rightarrow A^{p}\}_{p\in \Lambda_{\B}}$ as in the
proof of Proposition \ref{BigProp}.

To show that $\Phi =\Phi _{\widetilde{\mathcal{F}}}$, we need to
show that for all $g\in L$ and $p\in \Lambda_{\B}$, the following
diagram commutes:
\begin{equation}
\begin{array}{ccc}
W^{p} & \overset{\Phi (g)}{\longrightarrow } & W^{p} \\
\downarrow &  & \downarrow \\
A^{p} & \overset{\theta _{p}(g)}{\longrightarrow } & A^{p}%
\end{array}
\label{uniquephi}
\end{equation}%
where the vertical maps are both $\widetilde{\phi }_{p}$. \ We prove this by
induction on $p_{\max }$.

Let $p\in \Lambda_{\B}$. \ If $p_{\max }=1$, then (\ref{uniquephi}%
) commutes by our choice of $\{\widetilde{\phi }_{p}:W^{p}\rightarrow A^{p}\}_{p_{\max }=1}$ above. \ Let $m>1$ and assume (\ref%
{uniquephi}) commutes for all $p$ such that $p_{\max }<m$. \ Suppose then that $%
p_{\max }=m$. \ Choose some $\sigma \in \Sigma $ such that $p\in
\Lambda_{\B}(\sigma)$. \ By Lemma \ref{peel}, there exist $p^{\prime
},p^{\prime \prime }\in \Lambda_{\B}(\sigma,m-1)$ such that $%
p=p^{\prime }+p^{\prime \prime }$. \ Consider the following cube:
\begin{equation*}
\xymatrix{
& W^{p^{\prime }}\otimes W^{p''} \ar[dl]\ar[rr]^{\Phi (g)\otimes \Phi (g)}\ar'[d][dd]
& & W^{p^{\prime }}\otimes W^{p''} \ar[dl]\ar[dd]
\\
W^{p} \ar[rr]^>>>>>>>>>>>>{\Phi (g)}\ar[dd]
& & W^{p} \ar[dd]
\\
& A^{p^{\prime }}\otimes A^{p^{\prime \prime }} \ar[dl]\ar'[r][rr]^>>>>>>>>>>>>>>>>>>>>>>>>>>>>>{\theta
_{p^{\prime }}(g)\otimes \theta _{p^{\prime \prime }}(g)}
& & A^{p^{\prime }}\otimes A^{p^{\prime \prime }} \ar[dl]
\\
A^{p} \ar[rr]^>>>>>>>>>>>>{\theta _{p}(g)}
& & A^{p}
}
\end{equation*}
The top square commutes since $\Phi $ satisfies condition (2) of Theorem \ref{ExtThm2}%
. \ The left and right squares commute by the compatibility of
$\mathcal{F}$. \ The bottom square commutes since the product on
$\mathcal{A}_{\sigma }$ intertwines the $L$-action. \ Finally, the
back square commutes by inductive hypothesis. \ Since all the maps
are surjective, we conclude that the front square commutes. \ This
completes the induction, and proves that $\Phi =\Phi
_{\widetilde{\mathcal{F}}}$.

This completes the proof of Theorem \ref{ExtThm2}.

\section{Proof of Corollary \ref{ExtCor}\label{section7}}
Let $T_{SL_{2}}$ be the torus of $SL_{2}$ consisting of diagonal
matrices. Let $T_{L}=T_{SL_{2}}\times \cdots \times T_{SL_{2}}$ be
the diagonal torus of $L$. From elementary representation theory of
$SL_{2}$ we know that irreducible $L$-modules decompose canonically
into one dimensional $T_{L}$ weight spaces.  Now let
$(\Phi,\calli{B})$ be the representation of $L$ afforded by Theorem
\ref{ExtThm2}. Then by part (1) of Theorem \ref{ExtThm2}, the
multiplicity spaces $W^p$, ($p \in \Lambda_{\B}$), decompose into
one dimensional spaces.

The decomposition is canonical.  Indeed, $T_{L}$ is the unique torus
of $L$ containing $T_{SL_{2}}$. Moreover, the choice of the torus
$T_{SL_{2}}$ is induced by our choice of torus of $Sp_{n}$, i.e.
$$
T_{SL_{2}}=SL_{2} \cap T_{C_{n}}.
$$
Therefore the decomposition of $W^{\lambda / \mu}$ into one
dimensional spaces depends only the choice of torus of $Sp_{n}$.  We
now make this decomposition more precise.

\begin{lemma}
\label{basislemma}Let $\lambda / \mu \in \Lambda_{\B}$. \ The weight
spaces of $T_{L}$ on $W^{\lambda / \mu}$ are indexed by the set
$\{\gamma \in \Lambda_{n}:\mu <\gamma <\lambda ^{+}\}$. \ An element
$\gamma $ corresponds to the weight
\begin{equation*}
(t_{1},...,t_{n})\mapsto \dprod\limits_{i=1}^{n}t_{i}^{2\gamma
_{i}-x_{i}-y_{i}}\text{.}
\end{equation*}%
where $h(\lambda / \mu)=(x_{1},y_{1},...,x_{n},y_{n})$.
\end{lemma}

\begin{proof}
Set $p=\lambda / \mu$.  Since $W^{p}$ is isomorphic as an $L$-module
to $\bigotimes_{i=1}^{n}F^{r_{i}(p)}$, the
weight spaces of $T_{L}$ on $W^{p}$ are indexed by%
\begin{equation*}
\{(j_{1},...,j_{n}):0\leq j_{i}\leq r_{i}(p)\text{ for }i=1,...,n\}\text{.}
\end{equation*}%
Indeed, such a sequence $(j_{1},...,j_{n})$ corresponds to the weight
\begin{equation*}
(t_{1},...,t_{n})\mapsto \dprod\limits_{i=1}^{n}t_{i}^{(-r_{i}(p)+2j_{i})}%
\text{.}
\end{equation*}%
Now there is a one-to-one correspondence between
\begin{equation*}
\{(j_{1},...,j_{n}):0\leq j_{i}\leq r_{i}(p)\text{ for }i=1,...,n\}
\end{equation*}
and
\begin{equation*}
\{(\gamma _{1},...,\gamma _{n}):y_{i}\leq \gamma _{i}\leq
x_{i}\text{ for }i=1,...,n\}
\end{equation*}%
given by $(j_{1},...,j_{n})\mapsto (j_{1}+y_{1},...,j_{n}+y_{n})$. \
Therefore, by Lemma \ref{simplelemma}, the weight spaces of $T_{L}$ on $W^{p}$ are indexed
by $\{\gamma \in \Lambda_{n}:\mu <\gamma <\lambda ^{+}\}$. \ Unwinding these identifications, we see that a pattern $\gamma $ corresponds
to the weight
\begin{equation*}
(t_{1},...,t_{n})\mapsto
\dprod\limits_{i=1}^{n}t_{i}^{-(x_{i}+y_{i})+2\gamma _{i}}\text{.}
\end{equation*}
\end{proof}

For $\lambda / \mu \in \Lambda_{\B}$ and $\gamma \in \{\gamma \in
\Lambda_{n}:\mu <\gamma <\lambda ^{+}\}$, let $W^{\lambda / \gamma /
\mu}$ be the $T_{L}$ weight space of $W^{\lambda / \mu}$
corresponding to the weight $\gamma$ by the above lemma.  Then we
obtain a decomposition of $W^{\lambda / \mu}$ into one dimensional
weight spaces:
$$
W^{\lambda / \mu}=\bigoplus_{\substack{ \gamma \in \Lambda_{n} \\
\mu <\gamma <\lambda ^{+}}}W^{\lambda / \gamma / \mu}.
$$
This decomposition is canonical in the sense that it only depends on
the choice of torus $T_{C_{n}} \subset Sp_{2n}$.  This completes the
proof of Corollary \ref{ExtCor}.


\begin{thebibliography}{99}
\bibitem[GZ50]{GZ1} Gel'fand, I. M.; Zetlin, M. L. Finite-dimensional representations of the group of unimodular matrices. (Russian) \textit{Doklady Akad. Nauk SSSR (N.S.)}  71,  (1950). 825-828.
\bibitem[GW09]{GW} Goodman, Roe; Wallach, Nolan R. Symmetry, Representations, and Invariants. \textit{Graduate Texts in Mathematics}, 255. Springer, Dodrecht, 2009.
\bibitem[HTW08]{HTW} Howe, Roger; Tan, Eng-Chye; Willenbring, Jeb Reciprocity Algebras and Branching for Classical Symmetric Pairs, \textit{Groups and Analysis - the Legacy of Hermann Weyl}, London Mathematical Society Lecture Notes, 354, Cambridge University Press, 2008.
\bibitem[KY10]{KY} Kim, Sangjib; Yacobi, Oded, A basis for the symplectic
group branching algebra. Preprint 2010.
\bibitem[Mol99]{M2} Molev, A. I. A basis for representations of symplectic Lie algebras.  \textit{Comm. Math. Phys.}  201  (1999),  no. 3, 591-618.
\bibitem[Vin95]{Vin} Vinberg, Ernest B. The asymptotic semigroup of a semisimple Lie group.  \textit{Semigroups in algebra, geometry and analysis} (Oberwolfach, 1993),  293-310, de Gruyter Exp. Math., 20, de Gruyter, Berlin, 1995.
\bibitem[WY]{WY} Wallach, Nolan; Yacobi, Oded, A multiplicity formula for tensor products of $SL_{2}$ modules and an explicit $Sp_{2n}$ to $Sp_{2n-2} \times Sp_{2}$ branching formula, \textit{Contemp. Math.}, American Mathematical Society, Providence, R.I., 2009, (to appear).
\bibitem[Zh62]{Z0} Zhelobenko, D.P. The classical groups.  Spectral analysis of their finite dimensional representations, Russ. Math. Surv. 17 (1962), 1-94.
\bibitem[Zh73]{Z} Zhelobenko, D. P. Compact Lie groups and their representations. Translated from the Russian by Israel Program for Scientific Translations. \textit{Translations of Mathematical Monographs}, Vol. 40. American Mathematical Society, Providence, R.I., 1973.
\end{thebibliography}
\end{document}